\pdfoutput=1

\newcommand*{\rom}[1]{\expandafter\@slowromancap\romannumeral #1@}

\newcommand{\G}{\mathbb{G}}
\newcommand{\Gs}{\mathbb{G}_{\text{syl}}}
\newcommand{\A}{\mathbb{A}}
\newcommand{\AU}{\A U}
\newcommand{\As}{\A_{\text{syl}}}
\newcommand{\AG}{\A \G}
\newcommand{\KG}{K \G}
\newcommand{\AGs}{\A \G_{\text{syl}}}
\newcommand{\Af}{\A_{\text{fus}}}
\newcommand{\F}{\mathbb{F}}
\newcommand{\AF}{\A \F_p}

\newcommand{\Po}{\mathbb{P}}
\newcommand{\AP}{\A \Po}

\newcommand{\Z}{\mathbb{Z}}

\newcommand{\al}{\alpha}

\newcommand{\lra}[1]{\overset{#1}{\longrightarrow}}

\newcommand{\Prod}[1]{\underset{#1}{\prod}}

\newcommand{\cE}{\mathcal{E}}
\newcommand{\cF}{\mathcal{F}}
\newcommand{\cG}{\mathcal{G}}

\newcommand{\Sinf}{\Sigma^{\infty}}
\newcommand{\Sinfp}{\Sigma^{\infty}_{+}}
\newcommand{\Sinfpp}{\hat{\Sigma}^{\infty}_{+}}
\newcommand{\Sinfh}{\hat{\Sigma}^{\infty}}
	
\date{\today}

\documentclass[10pt]{amsart}
\usepackage{amssymb,amsfonts,amsthm,amsmath,verbatim,lscape,color,tensor,url}
\usepackage[margin=1.5in]{geometry}
\usepackage[enableskew]{youngtab}
\usepackage[all]{xy}

\usepackage{tikz}
\usetikzlibrary{arrows,matrix}
\tikzset{arrow/.style={semithick,>=stealth',shorten >=1pt,shorten <=1pt}}

\usepackage{mathtools}

\usepackage{todonotes}

\usepackage[pdftex,     
            plainpages=false,   
            breaklinks=true,    
            pdftitle={A formula for p-completion by way of the Segal conjecture},
            pdfauthor={Sune Precht Reeh, Tomer M. Schlank, Nathaniel Stapleton}
           ]{hyperref}

\usepackage[abbrev,msc-links,alphabetic]{amsrefs}
\makeatletter
\let\calc@author@part\calc@author@part@
\let\@suffix@format\@arabic
\let\append@label@year\@empty
\def\eprint#1{\@eprint#1 }
\def\@eprint #1:#2 {%
    \ifthenelse{\equal{#1}{arXiv}}%
        {\href{http://front.math.ucdavis.edu/#2}{arXiv:#2}}%
        {\href{#1:#2}{#1:#2}}%
}
\makeatother

\theoremstyle{definition}

\DeclareMathOperator{\colim}{colim}

\DeclareMathOperator{\Id}{Id}

\DeclareMathOperator{\Tr}{Tr}

\DeclareMathOperator{\Sp}{Sp}
\DeclareMathOperator{\Top}{Top}
\DeclareMathOperator{\Ho}{Ho}

\newtheorem{theorem}{Theorem}[section]
\newtheorem{corollary}[theorem]{Corollary}
\newtheorem{proposition}[theorem]{Proposition}
\newtheorem{lemma}[theorem]{Lemma}
\newtheorem{remark}[theorem]{Remark}
\newtheorem{definition}[theorem]{Definition}
\newtheorem{example}[theorem]{Example}

\makeatletter
\ifx\SK@label\undefined\let\SK@label\label\fi
 \let\your@thm\@thm
 \def\@thm#1#2#3{\gdef\currthmtype{#3}\your@thm{#1}{#2}{#3}}
 \def\mylabel#1{{\let\your@currentlabel\@currentlabel\def\@currentlabel
  {\currthmtype~\your@currentlabel}
 \SK@label{#1@}}\label{#1}}
 
\makeatother

\DeclareMathOperator{\id}{id}
\newcommand{\upperRomannumeral}[1]{\uppercase\expandafter{\romannumeral#1}}

\renewcommand{\phi}{\varphi}

\DeclarePairedDelimiter{\abs}{\lvert}{\rvert}

\begin{document}
\title{A formula for $p$-completion by way of the Segal conjecture}
\author[S. P. Reeh]{Sune Precht Reeh}
\author[T. M. Schlank]{Tomer M. Schlank}
\author[N. Stapleton]{Nathaniel Stapleton}

\maketitle

\begin{abstract}
The Segal conjecture describes stable maps between classifying spaces in terms of (virtual) bisets for the finite groups in question. Along these lines, we give an algebraic formula for the p-completion functor applied to stable maps between classifying spaces purely in terms of fusion data and Burnside modules.
\end{abstract}


\section{Introduction}

The $p$-completion of the classifying spectrum of a finite group is determined by the data of the induced fusion system on a Sylow $p$-subgroup. That is, if $G$ is a finite group, $S \subset G$ is a Sylow $p$-subgroup and $\cF_G$ is the fusion system on $S$ determined by $G$, then there is an equivalence of spectra
\[
(\Sinf BG)^{\wedge}_p \simeq \Sinf B\cF_G,
\]
where $B\cF_G$ is the classifying space associated to the fusion system (see Section \ref{secFusionSystems}).

The solution to the Segal conjecture provides an algebraic description of the homotopy classes of maps between suspension spectra of finite groups in terms of Burnside modules. In \cite{Ragnarsson}, a Burnside module between saturated fusions systems is defined. It is a submodule of the $p$-complete Burnside module between the Sylow $p$-subgroups that is characterized in terms of the fusion data. It is shown that this submodule captures the stable homotopy classes of maps between the $p$-completions of suspension spectra of finite groups. The $p$-completion functor induces a natural map of abelian groups
\[
[\Sinfp BG, \Sinfp BH] \rightarrow [(\Sinfp BG)^{\wedge}_{p}, (\Sinfp BH)^{\wedge}_{p}].
\]
In this paper, we give an algebraic description of this map in terms of fusion data.

Let $G$ and $H$ be finite groups. The proof of the Segal conjecture establishes a canonical natural isomorphism
\[
A(G,H)^{\wedge}_{I_G} \cong [\Sinfp BG, \Sinfp BH]
\]
between the Burnside module $A(G,H)$ of finite $(G,H)$-bisets with free $H$-action completed at the augmentation ideal $I_G$ of the Burnside ring $A(G)$ and the stable homotopy classes of maps between $BG$ and $BH$. Fix a prime $p$ and Sylow $p$-subgroups $S$ and $T$ of $G$ and $H$ respectively. Let $\cF_G$ and $\cF_H$ be the fusion systems on the fixed Sylow $p$-subgroups determined by $G$ and $H$. It follows from \cite{BLO2} that there are canonical (independent of the choice of Sylow $p$-subgroup) equivalences of spectra
\[
(\Sinf BG)^{\wedge}_p \simeq \Sinf B\cF_G \text{ and } (\Sinfp BG)^{\wedge}_p \simeq \Sinf B\cF_G \vee (S^0)^{\wedge}_{p} \simeq (\Sinfp B\cF_G)^{\wedge}_{p}.
\]
The Burnside module for the fusion systems $\cF_G$ and $\cF_H$, as defined in \cite{Ragnarsson}, is the submodule
\[
\AF(\cF_G, \cF_H) \subset A(S,T)^{\wedge}_{p}
\]
consisting of ``fusion stable" $(S,T)$-bisets. Stability is defined entirely in terms of the fusion data. The restriction of a $(G,H)$-biset ${}_{G}X_{H}$ to the Sylow $p$-subgroups $S$ and $T$ induces a map of Burnside modules
\[
A(G,H) \rightarrow A(S,T) \hookrightarrow A(S,T)^{\wedge}_{p}
\]
sending ${}_{G}X_{H}$ to ${}_{S}X_{T}$. This map lands inside the stable elements:
\[
\xymatrix{A(G,H) \ar[r] \ar@{-->}[dr] & A(S,T)^{\wedge}_{p} \\ & \AF(\cF_G,\cF_H). \ar@{^{(}->}[u]}
\]
We will write ${}_{\cF_G}X_{\cF_H}$ for ${}_{S}X_{T}$ viewed as an element of $\AF(\cF_G,\cF_H)$. Corollary 9.4 of \cite{RagnarssonStancu} essentially produces an isomorphism
\[
\AF(\cF_G, \cF_H) \cong [(\Sinfp BG)^{\wedge}_{p}, (\Sinfp BH)^{\wedge}_{p}].
\]
Using the $p$-completion functor
\[
(-)^{\wedge}_{p} \colon [\Sinfp BG, \Sinfp BH] \rightarrow [(\Sinfp BG)^{\wedge}_{p}, (\Sinfp BH)^{\wedge}_{p}]
\]
we can form the composite
\[
\widehat{(-)} \colon A(G,H) \rightarrow A(G,H)^{\wedge}_{I_G} \cong [\Sinfp BG, \Sinfp BH] \lra{(-)^{\wedge}_{p}} [(\Sinfp BG)^{\wedge}_{p}, (\Sinfp BH)^{\wedge}_{p}] \cong \AF(\cF_G, \cF_H).
\]
It is natural to ask for a completely algebraic description of this map in terms of bisets. This is not just the restriction map, an extra ingredient is needed. Let $_{T}H_{T}$ be the underlying set of $H$ acted on the left and right by $T$. Since this is the restriction of $_{H}H_{H}$, it is stable so we may consider it as an element $_{\cF_H}H_{\cF_H} \in \AF(\cF_H,\cF_H)$.  It is invertible as $|H/T|$ is prime to $p$.

\begin{theorem} \label{mainone} (Theorem \ref{mainthm})
The ``completion" map
\[
\widehat{(-)} \colon A(G,H) \lra{} \AF(\cF_G,\cF_H)
\]
is given by
\[
\widehat{_{G} X_{H}} = (_{\cF_H}H_{\cF_H})^{-1} \circ (_{\cF_G}X_{\cF_H}) = {}_{\cF_G}X \times_T H^{-1}_{\cF_H}.
\]
\end{theorem}

Thus we have a commutative diagram
\[
\xymatrix{[\Sinfp BG, \Sinfp BH] \ar[r]^-{(-)^{\wedge}_{p}} & [(\Sinfp BG)^{\wedge}_{p}, (\Sinfp BH)^{\wedge}_{p}] \ar[d]^{\cong} \\ A(G,H) \ar[u] \ar[r]^-{\widehat{(-)}} & \AF(\cF_G, \cF_H),}
\]
where $\widehat{(-)}$ is given by the formula in the theorem above.

Along the way to proving Theorem \ref{mainone}, we review the theory of Burnside modules, spectra, fusion systems, Burnside modules for fusion systems, and $p$-completion as well as proving a few folklore results. We prove that the suspension spectrum of the $p$-completion of the classifying space of a finite group is the same as the $p$-completion of the classifying spectrum
\[
\Sinf (BG^{\wedge}_{p}) \simeq (\Sinf BG)^{\wedge}_{p}.
\]
We also show that the $p$-completion map induces an isomorphism
\[
[\Sinfp BG, \Sinfp BH]^{\wedge}_{p} \lra{\cong} [(\Sinfp BG)^{\wedge}_{p}, (\Sinfp BH)^{\wedge}_{p}]
\]
and give explicit formulas for $(_{\cF_H}H_{\cF_H})^{-1}$ that aid computation.

Let $\AG$ be the Burnside category with finite groups as objects and Burnside modules $A(-,-)$ as morphism sets. Then $\widehat{(-)}$ does not directly define a functor $\AG \to \AF$ on objects because the fusion system $\cF_G$ depends on the choice of a Sylow $p$-subgroup in $G$, even if different choices give isomorphic fusion systems.

Let $\Gs$ be the category of finite groups with a chosen Sylow $p$-subgroup, and let $\AGs$ be the associated Burnside category. Then $\widehat{(-)}$ gives a well-defined functor $\AGs \to \AF$ fitting into a commutative diagram (with notation introduced in Appendix \ref{sectionCategories}):
\begin{equation}\label{eqDiagramOfCats}
\xymatrix{\G \ar[r]^{\A}  &\AG \ar[r]^-{\al}& \Ho(\Sp) \ar[dd]^-{(-)^{\wedge}_{p}} \\ \Gs \ar[u]^{U} \ar[r]^{\As} \ar[d]_{F} & \AGs \ar[u]^-{\AU}_-{\simeq}  \ar[d]_-{\widehat{(-)}} &  \\ \F \ar[r]_{\Af} &\AF \ar[r]^-{\beta} & \Ho(\Sp_p).}
\end{equation}
In Appendix \ref{sectionCategories}, we work out this diagram and recall the categories and functors involved.

\subsection*{Outline}
Section \ref{secPreliminaries} recalls the preliminaries for the paper and touches on the following topics in order: Bisets and Burnside modules, stable maps and the Segal conjecture for finite groups, the effect of disjoint base points for the Segal conjecture, fusion systems and their Burnside modules, and finally p-completion of classifying spectra of finite groups.

In Section \ref{secAFormulaForPCompletion} we explore the equivalence between $p$-completed classifying spectra of finite groups and classifying spectra of fusion systems from the view point of bisets, and we provide a proof of Theorem \ref{mainone} as Theorem \ref{mainthm}.

In Appendix \ref{sectionCategories} we provide a summary of how the functor $\widehat{(-)}$ of Theorem \ref{mainone} fits into the commutative diagram \eqref{eqDiagramOfCats}.

\subsection*{Acknowledgments}
All three authors would like to thank the Max Planck Institute for Mathematics and the Hausdorff Research Institute for Mathematics for their hospitality and support. While working on this paper the first author was funded by the Independent Research Fund Denmark (DFF–4002-00224) and later on by BGSMath and the María de Maeztu Programme (MDM–2014-0445). The second and third author were jointly supported by the US-Israel Binational Science Foundation under grant 2018389. The second author was partially supported by the Alon Fellowship and ISF 1588/18. The third author was partially supported by NSF grant DMS-1906236 and the SFB 1085 \emph{Higher Invariants} at the University of Regensburg. The first author would like to thank his working group at HIM for being patient listeners to discussions about bisets, $I$-adic topologies, and $p$-completions.

\section{Preliminaries}\label{secPreliminaries}
The purpose of this section is to recall definitions and results that are relevant to this paper. We review Burnside modules, spectra, fusion systems, and notions of $p$-completion. On top of this, we prove a few folklore results regarding $p$-completion.

\subsection{Burnside modules} \label{sec:burnside}

Let $G$ and $H$ be finite groups. Recall that a finite $(G,H)$-biset $X$ is a finite set equipped with a left action of $G$ and a right action of $H$ with the property that these actions commute.

\begin{definition}
Let $\AG$ be the Burnside category of finite groups. The objects are finite groups. The morphism set between two groups $G$ and $H$, $\AG(G,H)$, is the Grothendieck group of isomorphism classes of finite $(G,H)$-bisets with free $H$-action and disjoint union as addition. We will refer to the elements in $\AG(G,H)$ as virtual bisets. Given a third group $K$, the composition map
\[
\AG(H,K) \times \AG(G,H) \rightarrow \AG(G,K)
\]
is induced by the map sending an $(H,K)$-biset $Y$ and a $(G,H)$-biset $X$ to the coequalizer $X \times_H Y$. The composition map is bilinear. The identity map in $\AG(G,G)$ is the $(G,G)$-biset $G$ with $G$ acting by left and right multiplication, respectively.
\end{definition}

There is a canonical basis of $\AG(G,H)$ as a $\Z$-module given by the isomorphism classes of transitive $(G,H)$-bisets. These bisets are of the form
\[
G \times_{K}^{\phi} H = (G \times H) \big / (gk,h) \sim (g,\phi(k)h),
\]
where $K \subseteq G$ is a subgroup of $G$ and $\phi \colon K \rightarrow H$ is a group homomorphism. If we precompose $\phi$ with a conjugation map in $G$ and postcompose $\phi$ with a conjugation map in $H$, the construction above gives rise to an isomorphic $(G,H)$-biset. We will denote the isomorphism classes of these $(G,H)$-bisets by $[K,\phi]_{G}^{H}$ or just $[K,\phi]$ when $G$ and $H$ are clear from context. It is also common to denote a virtual biset $X \in \AG(G,H)$ as ${}_{G}X_{H}$ when $G$ and $H$ are not clear from context.

There is also an ``unpointed" version of the category $\AG$, where we remove the part of $\AG(G,H)$ that is seen by the projection $H\to e$ to the trivial group:

\begin{definition}
Let $\KG$ be the category with objects finite groups and morphism sets given by
\[
\KG(G,H) = \ker(\AG(G,H) \lra{\epsilon} \AG(G,e)),
\]
where $\epsilon({}_{G}X_{H}) = {}_G(X/H)_{e}$.

The identity morphism in $\KG(G,G)$ is the virtual biset 
\[[G,i_G]-[G,0] = (G\times_G G) - (G/G\times_e G)\] where $0\colon G\to G$ is the trivial map sending every element to the neutral element. The virtual biset $[G,i_G]-[G,0]$ is idempotent in $\AG(G,G)$, and 
\[\KG(G,H)=([G,i_G]-[G,0])\ \AG(G,H)\ ([H,i_H]-[H,0]) =\AG(G,H)\ ([H,i_H]-[H,0]).\]
\end{definition}

Let $A(G)$ be the Burnside ring of $G$: Additively $A(G) = \AG(G,e)$, but the multiplicative structure comes from the cartesian product of left $G$-sets (with diagonal action). This ring may be identified with the (commutative) subring $A^{\text{char}}(G) \subset \AG(G,G)$ spanned by the $(G,G)$-bisets of the form $G \times_K G = [K,i_K]$ (known as the semicharacteristic $(G,G)$-bisets), where $K \subseteq G$ is a subgroup and $i_K \colon K \subset G$ is the inclusion. The identification of $A^{\text{char}}(G)$ with $A(G)$ is given by the composite
\[
A^{\text{char}}(G) \subset \AG(G,G) \lra{\epsilon} \AG(G,e).
\]
The inverse isomorphism sends $G/H \in A(G)$ to $G/H \times G = G \times_H G \in A^{\text{char}}(G)$, where the left action of $G$ on $G/H \times G$ is diagonal. Since $\AG(G,H)$ is a left $\AG(G,G)$-module, it is also a left $A(G)$-module. Finally, let $I_G \subset A(G)$ be the kernel of the augmentation $A(G) \rightarrow \Z$ sending a $G$-set $X$ to its cardinality, and let $p+I_G$ denote the sum of the ideals $I_G$ and $pA(G)$.

\begin{lemma}[\cite{MayMcClure}]\label{lemmaIpadic}
If $G$ is a $p$-group of order $p^n$, then $I_G^{n+1}\subseteq pI_G$. Furthermore, the $I_G$-adic and $p$-adic topologies on $\KG(G,H)$ coincide and the $(p+I_G)$-adic and $p$-adic topologies on $\AG(G,H)$ coincide.
\end{lemma}

\begin{proof}
Lemma 5 in \cite{MayMcClure} states that the $I_G$-adic and $p$-adic topologies on $\KG(G,H)$ coincide, and $I_G^{n+1}\subseteq pI_G$ is the key part of the proof of said lemma. Finally, we have $(p+I_G)^{n+1} \subseteq p A(G)$ by the previous formula, and clearly $pA(G)\subseteq p+I_G$, so the $(p+I_G)$-adic and $p$-adic topologies coincide on any $A(G)$-module, hence in particular on $\AG(G,H)$.
\end{proof}

\begin{remark}
When $G$ is a $p$-group the ideal $p+I_G$ simply consists of the virtual $G$-sets $X\in A(G)$ with $p \mid \abs X$, i.e. the kernel of the mod-$p$ augmentation $A(G)\to \F_p$.
\end{remark}

\subsection{Spectra}

\begin{definition}
Let $\Ho(\Sp)$ be the homotopy category of spectra.
\end{definition}

For a pointed space $X$, let $\Sinf X$ be the suspension spectrum of $X$. Given another pointed space $Y$, we let
\[
[\Sinf X, \Sinf Y]
\]
be the abelian group of homotopy classes of stable maps. For $X$ unpointed, let $\Sinfp X$ be the suspension spectrum of $X$ with a disjoint basepoint. When $G$ is a finite group we will write $\Sinfp BG$ for the suspension spectrum of the classifying space $BG$ with a disjoint basepoint and $\Sinf BG$ for the suspension spectrum of $BG$ using $Be \rightarrow BG$ as the basepoint.

There is a canonical map
\[
\AG(G,H) \rightarrow [\Sinfp BG,\Sinfp BH]
\]
sending $[K,\phi]_{G}^{H}$ to the composite
\[
\Sinfp BG \lra{\text{Tr}} \Sinfp BK \xrightarrow{\Sinfp B\phi} \Sinfp BH,
\]
where $\text{Tr}$ is the transfer along the map $BK \to BG$, which is equivalent to a finite cover. This canonical map extends to a natural transformation of functors from $\AG^{\mathrm{op}} \times \AG$ to the category of abelian groups. This map was intensely studied over several decades culminating in the following theorem.

\begin{theorem}[\cite{CarlssonSegal}, \cite{AGM}, \cite{LewisMayMcClure}, ``The Segal conjecture"] \label{SegalConj}
There is a canonical isomorphism of commutative rings
\[
[\Sinfp BG,S^0] \cong A(G)^{\wedge}_{I_G}
\]
and a canonical isomorphism of $A(G)^{\wedge}_{I_G}$-modules
\[
[\Sinfp BG,\Sinfp BH] \cong \AG(G,H)^{\wedge}_{I_G}.
\]
The first isomorphism is a natural isomorphism of functors from the opposite category of finite groups to the category of commutative rings. The second isomorphism is a natural isomorphism of functors from $\AG^{\mathrm{op}} \times \AG$ to the category of abelian groups.
\end{theorem}

Several canonical isomorphisms follow from this theorem (see \cite{MayStableMaps}). For instance, \cite[Theorem 13]{MayStableMaps} produces a canonical isomorphism
\[
[\Sinf BG,\Sinf BH] \cong \KG(G,H)^{\wedge}_{I_G}.
\]
For $S$ a $p$-group, due to Lemma \ref{lemmaIpadic}, we have isomorphisms
\[
[ \Sinf BS, \Sinf BH] \cong \KG(S,H)^{\wedge}_{p}
\]
and
\[
[\Sinfp BS,\Sinfp BH]^{\wedge}_{p} \cong \AG(S,H)^{\wedge}_{p}.
\]
According to \cite{RagnarssonStancu}*{Proposition 9.2}, we can also relax the $p$-completion of stable maps in the last formula:
\[
[\Sinfp BS,\Sinfp BH] \cong \{X\in \AG(S,H)^{\wedge}_{p} \mid \abs{X}/\abs{S}\in \Z\},
\]
when $S$ is a $p$-group.

\subsection{Base points and idempotents} \label{remarkDisjointBasepoints}

The splitting 
\[
\Sinfp BG \simeq \Sinf BG \vee S^0 \simeq \Sinf BG \times S^0
\]
corresponds to a pair of complementary idempotents in $\AG(G,G)^\wedge_{I_G}$. In this section we will consider these idempotents and carefully work out how to go back and forth between $\Sinfp BG$ and $\Sinf BG$ while working with bisets. 

The projection $\Sinfp BG\to S^0$ and inclusion $S^0\to \Sinfp BG$ are induced by the group homomorphisms $0\colon G\to e$ and $i_e\colon e\to G$, respectively. Hence the idempotent of $\AG(G,G)^\wedge_{I_G}$ that splits off $S^0$ as a summand of $\Sinfp BG$, is the biset $[G,0]_G^e\times_e [e,i_e]_e^G = [G,0]_G^G$.

The complementary idempotent, $[G,i_G]-[G,0]\in \AG(G,G)^\wedge_{I_G}$, then splits of the summand $\Sinf BG$ from $\Sinfp BG$.

Multiplying with $[H,0]_H^H$ from the right, takes any $(G,H)$-biset $X$ to $(X/H)\times_e H \in \AG(G,H)^\wedge_{I_G}$. This is the map $\epsilon\colon \AG(G,H)^\wedge_{I_G} \to \AG(G,e)^\wedge_{I_G}$ followed by induction of bisets back up to $H$. Multiplying with  $[H,0]_H^H$ from the right corresponds to projecting onto $S^0$ and then including $S^0$ back into $\Sinfp BH$.

Multiplying $\AG(G,H)^\wedge_{I_G}$ with the complementary idempotent $([H,i_H]-[H,0])$ from the right gives the kernel of $\epsilon\colon \AG(G,H)^\wedge_{I_G}\to \AG(G,e)^\wedge_{I_G}$,
\[\AG(G,H)^\wedge_{I_G}([H,i_H]-[H,0]) = \KG(G,H)^\wedge_{I_G}.\]

A map $\Sinfp BG \rightarrow \Sinfp BH$ is determined by four maps between the summands. Algebraically, this corresponds to the splitting of $\AG(G,H)^{\wedge}_{I_G}$ by applying the idempotents $[G,0], ([G,i_G]-[G,0]) \in \AG(G,G)$ and $[H,0], ([H, i_H] - [H,0])  \in \AG(H,H)$ from the left and right respectively. Explicitly, we have the following isomorphisms:
\begin{align*}
[\Sinfp BG, \Sinfp BH] &\cong [\Sinf BG, \Sinf BH] \oplus [S^0,\Sinf BH] \oplus [\Sinf BG, S^0] \oplus [S^0,S^0],
\\ [\Sinf BG, \Sinf BH] &\cong ([G,i_G]-[G,0]) \AG(G,H)^{\wedge}_{I_G} ([H, i_H] - [H,0]),
\\ [S^0,\Sinf BH] &\cong [G,0] \AG(G,H)^{\wedge}_{I_G} ([H, i_H] - [H,0]) \cong 0,
\\ [\Sinf BG, S^0] &\cong ([G,i_G]-[G,0]) \AG(G,H)^{\wedge}_{I_G} [H,0] \cong \{X \times_e H\mid X\in A(G)^\wedge_{I_G}, \abs{X} = 0\},
\\ [S^0,S^0] &\cong [G,0] \AG(G,H)^{\wedge}_{I_G} [H,0] \cong \{a\cdot[G,0]_G^H \mid a\in \Z\}.
\end{align*}
The statement that $\Sinf BH$ is connected, so that $[S^0,\Sinf BH]=0$, corresponds to the algebraic fact $[G,0] \AG(G,H)^{\wedge}_{I_G} ([H, i_H] - [H,0]) \cong 0$, which is easily confirmed for each basis element $[K,\phi]_G^H\in \AG(G,H)^{\wedge}_{I_G}$:
\begin{align*}
[G,0]_G^G \times_G [K,\phi]_G^H \times_H ([H, i_H]_H^H - [H,0]_H^H) &= \abs{G/K}\cdot [G,0]_G^H \times_H ([H, i_H]_H^H - [H,0]_H^H)
\\ &=\abs{G/K} \cdot ([G,0]_G^H - [G,0]_G^H)
\\ &= 0.
\end{align*}

Further, this implies that
\[[\Sinf BG, \Sinf BH] \cong [\Sinfp BG, \Sinf BH] \cong \AG(G,H)^{\wedge}_{I_G} ([H, i_H] - [H,0])= \KG(G,H)^\wedge_{I_G}.\]
Given any map $f\colon \Sinfp BG\to \Sinfp BH$ represented by a virtual biset $X\in \AG(G,H)^\wedge_{I_G}$, we can find the part of $f$ that goes from $\Sinf BG$ to $\Sinf BH$ by the formula
\[X\times_H ([H, i_H] - [H,0]) = ([G,i_G]-[G,0])\times_G X\times_H ([H, i_H] - [H,0]) \in \KG(G,H)^\wedge_{I_G}.\]
Consequently, most of the results in this paper about $[\Sinfp BG, \Sinfp BH]$ can be converted into results about $[\Sinf BG, \Sinf BH]$ by multiplying with $([H, i_H] - [H,0])$ from the right.

\subsection{Fusion systems}\label{secFusionSystems}
We recall the very basics of the definition of a saturated fusion system. For additional details see \cite{ReehIdempotent}*{Section 2}, \cite{RagnarssonStancu}*{Section 2} or \cite{AKO}*{Part I}. We also discuss the construction of the classifying spectrum of a fusion system.

\begin{definition}
A fusion system on a finite $p$-group $S$ is a category $\cF$ with the subgroups of $S$ as objects and where the morphisms $\cF(P,Q)$ for $P,Q\leq S$ satisfy
\begin{itemize}
\item[(i)] Every morphism $\phi\in \cF(P,Q)$ is an injective group homomorphism $\phi\colon P\to Q$.
\item[(ii)] Every map $\phi\colon P\to Q$ induced by conjugation in $S$ is in $\cF(P,Q)$.
\item[(iii)] Every map $\phi\in\cF(P,Q)$ factors as $P\xrightarrow{\phi} \phi(P) \hookrightarrow Q$ in $\cF$ and the inverse isomorphism $\phi^{-1}\colon \phi(P)\to P$ is also in $\cF$.
\end{itemize}
A \emph{saturated} fusion system satisfies some additional axioms that we will not go through as they play no direct role in this paper.

We say that two subgroups, $P, Q \leq S$, are conjugate in $\cF$ if $Q = \phi(P)$ for some morphism $\phi$ in the fusion system $\cF$.

Given fusion systems $\cF_1$ and $\cF_2$ on $p$-groups $S_1$ and $S_2$, respectively, a group homomorphism $\phi\colon S_1\to S_2$ is said to be fusion preserving if whenever $\psi\colon P\to Q$ is a map in $\cF_1$, there is a corresponding map $\rho\colon \phi(P)\to \phi(Q)$ in $\cF_2$ such that $\phi|_{Q}\circ \psi = \rho\circ \phi|_P$. Note that each such $\rho$ is unique if it exists.
\end{definition}

\begin{example}
Whenever $G$ is a finite group with Sylow $p$-subgroup $S$, we associate a fusion system on $S$ denoted $\cF_G$. The maps in $\cF_G(P,Q)$ for subgroups $P,Q\leq S$ are precisely the homomorphisms $P\to Q$ induced by conjugation in $G$.
The fusion system $\cF_G$ associated to a group at a prime $p$ is always saturated.
\end{example}

Every saturated fusion system $\cF$ has a classifying spectrum originally constructed by Broto-Levi-Oliver in \cite{BLO2}*{Section 5}. The most direct way of constructing this spectrum is due to Ragnarsson and Stancu in \cite{Ragnarsson}*{Section 7} and \cite{RagnarssonStancu}*{Section 9.3}. From the data of a saturated fusion system $\cF$, they construct (see \cite{Ragnarsson}*{Definition 4.3}) an idempotent, called the characteristic idempotent, $\omega_\cF \in \AG(S,S)^{\wedge}_{p}$. Applying the Segal conjecture, this data is equivalent to a map of spectra
\[
\omega_\cF \colon \Sinfp BS \to \Sinfp BS.
\]
The classifying spectrum of $\cF$ is defined to be the mapping telescope
\[
\Sinfp B\cF = \colim ( \Sinfp BS \lra{\omega_{\cF}} \Sinfp BS \lra{\omega_{\cF}} \ldots).
\]
By construction $\Sinfp B\cF$ is a wedge summand of $\Sinfp BS$. The transfer map
\[
t \colon \Sinfp B\cF \rightarrow \Sinfp BS
\]
is the inclusion of $\Sinfp B\cF$ as a summand and the ``inclusion" map
\[
r \colon \Sinfp BS \rightarrow \Sinfp B\cF
\]
is the projection on $\Sinfp B\cF$. Thus $r \circ t = 1$ and $t\circ r=\omega_{\cF}$.

As remarked in Section 5 of \cite{BLO2}, the spectrum $\Sinfp B\cF$ constructed this way is in fact the suspension spectrum for the classifying space $B\cF$ defined in \cites{BLO2, Chermak}. One way to see this is to note that $H^*(B\cF,\F_p)$ coincides with $\omega_\cF \cdot H^*(BS,\F_p)$ as the $\cF$-stable elements, and by an argument similar to Proposition \ref{propCompleteBG} later on, the suspension spectrum of $B\cF$ is $H\F_p$-local.

\subsection{Burnside modules for fusion systems}

Fix a prime $p$.
\begin{definition}
Let $\AF$ be the Burnside category of saturated fusion systems. The objects in this category are saturated fusion systems $(\cF,S)$ over finite $p$-groups. Let $\cF_1$ and $\cF_2$ be saturated fusion systems on $p$-groups $S_1$ and $S_2$. The morphisms in $\AF$ between $(\cF_1,S_1)$ and $(\cF_2,S_2)$ are a certain submodule of the Burnside module (see \cite{ReehIdempotent}*{Definition 5.15})
\[
\AF(\cF_1,\cF_2) \subseteq \AF(S_1,S_2) = \AG(S_1,S_2)^{\wedge}_{p}.
\]
This is the submodule of the $p$-complete Burnside module $\AG(S_1, S_2)^{\wedge}_{p}$ consisting of left $\cF_1$-stable and right $\cF_2$-stable elements.
\end{definition}

Stability may be defined in two ways. We say that an element $X \in \AF(S_1, S_2)$ is left $\cF_1$-stable if $\omega_1 \circ X = X$ and right $\cF_2$-stable if $X \circ \omega_2 = X$, where $\omega_1$ and $\omega_2$ are the characteristic idempotents of $\cF_1$ and $\cF_2$. Algebraically, the definition is longer but more elementary. An $(S_1,S_2)$-biset $X$ is left $\cF_1$-stable if for all pairs of subgroups $P,Q \subset S_1$ and any isomorphism $\phi \colon P \cong_{\cF_1} Q$ in $\cF_1$ the $(P,S_1)$-sets ${}_PX_{S_1}$ and ${}_{P}^{\phi}X_{S_1}$ are isomorphic, where ${}_{P}^{\phi}X_{S_1}$ is the biset induced by restriction along
\[
P \lra{\phi} Q \subset S_1.
\]
Right stability is defined similarly. The Burnside module $\AF(\cF_1,\cF_2)$ is the $p$-completion of the Grothendieck group of left $\cF_1$-stable right $\cF_2$-stable $(S_1,S_2)$-bisets (\cite[Proposition 4.4]{ReehStableSets}). For short, we will call such left $\cF_1$-stable right $\cF_2$-stable elements $(\cF_1,\cF_2)$-stable or just stable if $\cF_1$ and $\cF_2$ are clear from context.

Since $\AF(S_1,S_2) = \AG(S_1,S_2)^\wedge_p$, the Burnside module $\AF(S_1,S_2)$ is the free $\Z_p$-module on bisets of the form $[K, \phi]_{S_1}^{S_2}$. Similarly, it follows from \cite[Proposition 5.2]{Ragnarsson} that $\AF(\cF_1,\cF_2)$ is a free $\Z_p$-module on basis elements denoted $[K,\phi]_{\cF_1}^{\cF_2}$, and given by
\[
[K,\phi]_{\cF_1}^{\cF_2} = \omega_{\cF_1} \circ [K,\phi]_{S_1}^{S_2} \circ \omega_{\cF_2},
\]
where $K\leq S_1$ and $\phi\colon K\to S_2$. Just as for a finite group, if we precompose $\phi$ with an isomorphism from $\cF_1$ and postcompose $\phi$ with an isomorphism from $\cF_2$, we get the same basis element in $\AF(\cF_1,\cF_2)$.

In order to clarify that a stable $(S_1,S_2)$-biset $X$ is being viewed as an element in $\AF(\cF_1, \cF_2)$ we will write ${}_{\cF_1}X_{\cF_2}$. Given a third saturated fusion system $\cF_3$ on $S_3$ and bisets ${}_{\cF_1}X_{\cF_2}$ and ${}_{\cF_2}Y_{\cF_3}$, we will denote the composite biset by
\[
{}_{\cF_1}X \times_{S_2} Y_{\cF_3} = ({}_{\cF_1}X_{\cF_2}) \times_{S_2} ({}_{\cF_2}Y_{\cF_3}).
\]

Given finite groups $G$ and $H$ with Sylow subgroups $S$ and $T$ and a $(G,H)$-biset $X$, we may restrict the $G$-action to $S$ and the $H$-action to $T$ to get an $(S,T)$-biset ${}_{S}X_{T}$. Let $\cF_G$ be the saturated fusion system associated to $G$ on $S$ and $\cF_H$ the saturated fusion system associated to $H$ on $T$. We leave it as an exercise to the reader to check that the restricted biset  ${}_{S}X_{T}$ is always a stable biset and so we may further consider it as an $(\cF_G, \cF_H)$-stable biset
\[
{}_{\cF_G}X_{\cF_H}.
\]

We turn our attention to Burnside rings for saturated fusion systems. Recall that the composite
\[
A^{\text{char}}(G) \rightarrow \AG(G,G) \rightarrow \AG(G,e) \cong A(G)
\]
of Section \ref{sec:burnside} is an isomorphism and identifies the Burnside ring $A(G)$ with the subring of $\AG(G,G)$ on the semicharacteristic $(G,G)$-bisets.

In the same way, there are two versions of the Burnside ring associated to a fusion system $\cF$ on a $p$-group $S$. The first, denoted $A(\cF)$, is the subring of $\cF$-stable elements of $A(S)$. The second is the subring of $A_{p}^{\text{char}}(\cF) \subseteq \AF(\cF,\cF) \subseteq \AF(S,S)$ consisting of $\cF$-semicharacteristic bisets. This is the $\Z_p$-submodule spanned by the basis elements of the form $[K,i_K]_{\cF}^{\cF}$. The identity element in $A_{p}^{\text{char}}(\cF)$ is the characteristic idempotent. The units of $A_{p}^{\text{char}}(\cF)$ are usually referred to as the $\cF$-characteristic elements, and each of them contains enough information to reconstruct $\cF$ (see \cite{RagnarssonStancu}*{Theorem 5.9}).

The commutative rings $A(\cF)$ and $A_{p}^{\text{char}}(\cF)$ may be canonically identified after $p$-completion, but it is useful to distinguish between the two. The (non-multiplicative) map $\epsilon \colon \AG(S,S) \rightarrow A(S)$ induces a map
\[
A_{p}^{\text{char}}(\cF) \subset \AF(\cF,\cF) \rightarrow A(\cF)^{\wedge}_{p},
\]
which is a ring isomorphism by Theorem D in \cite{ReehIdempotent}.

Let $I_\cF$ be the kernel of the augmentation map
\[
I_{\cF} = \ker(A(\cF) \rightarrow A(S) \rightarrow \Z).
\]
\begin{remark}\label{AFLocal}
The ring $A(\cF)^{\wedge}_{p}$ is clearly $p$-complete. It is also complete with respect to the maximal ideal $(p)+I_\cF$ of $A(\cF)$ and these ideals give the same topology. This follows immediately from Lemma \ref{lemmaIpadic}, which states that the ideals $(p)$ and $(p)+I_S$ give the same topology on $A(S)$.

As the completion of $A(\cF)$ with respect to the maximal ideal $(p)+I_\cF$ the ring $A(\cF)^{\wedge}_{p}$ is in fact complete local with maximal ideal $(p)+I_\cF$.
\end{remark}


\subsection{$p$-completion}

Let $E$ be a spectrum. There is a $p$-completion functor on spectra equipped with a canonical transformation
\[
E \rightarrow E^{\wedge}_{p}.
\]
This functor is given by Bousfield localization at the Moore spectrum $M\Z/p$ (see \cite{Bousfield2}*{Proposition 2.5}). When $E$ is connective, for instance if $E$ is the classifying spectrum of a finite group, then $E^{\wedge}_{p}$ is also the localization of $E$ at $H\F_p$. There is a natural equivalence
\[
(\Sinfp BG)^{\wedge}_{p} \simeq (\Sinf BG)^{\wedge}_{p} \vee (S^0)^{\wedge}_{p}.
\]
The arithmetic fracture square immediately implies that
\[
\Sinf BG \simeq \bigvee_p (\Sinf BG)^{\wedge}_p.
\]
If $S$ is a $p$-group, then $\Sinf BS \simeq (\Sinf BS)^{\wedge}_{p}$ so
\[
(\Sinfp BS)^{\wedge}_{p} \simeq \Sinf BS \vee (S^0)^{\wedge}_{p}.
\]
When the prime is clear from context and $X$ is a space, we will write
\[
\Sinfpp X = (\Sinfp X)^{\wedge}_{p}
\]
for the $p$-completion of the suspension spectrum with a disjoint basepoint and, if $X$ is pointed, $\Sinfh X$ for $(\Sinf X)^{\wedge}_{p}$, the $p$-completion of the suspension spectrum.


There are several notions of $p$-completion for spaces. These were developed in \cite{Bousfield}, \cite{BousfieldKan}, \cite{Sullivan1}, and \cite{Sullivan2}. For a space such as $BG$, these notions all agree and there is a simple relationship between the stable $p$-completion and the unstable $p$-completion. Since it was difficult to find a proof of this fact in the literature, we provide a complete proof. Note that, since this paper was first made available, this folklore result has also appeared in \cite{BB}.

\begin{proposition}\label{propCompleteBG}
Let $G$ be a finite group. There is a canonical equivalence
\[
\Sinf (BG^{\wedge}_{p}) \simeq (\Sinf BG)^{\wedge}_{p}.
\]
\end{proposition}
\begin{proof}
It follows from \cite[VII.4.3]{BousfieldKan} that the unstable homotopy groups $\pi_*(BG^{\wedge}_{p})$ are all finite $p$-groups.

This implies that the reduced integral homology groups $\widetilde{H\Z}_*(BG^{\wedge}_{p})$ are all finite $p$-groups: Let $\widetilde{BG^{\wedge}_{p}}$ be the universal cover. A Serre class argument with the Serre spectral sequence associated to the fibration
\[
\widetilde{BG^{\wedge}_{p}} \rightarrow BG^{\wedge}_{p} \rightarrow K(\pi_1(BG^{\wedge}_{p}),1)
\]
reduces this problem to the group homology of $\pi_1(BG^{\wedge}_{p})$ with coefficients in the integral homology of the fiber. The result follows from the fact that the reduced homology groups of the fiber are $p$-groups and that the integral group homology groups of $\pi_1(BG^{\wedge}_{p})$ are $p$-groups. Both of these are classical (see \cite[Chapter 10]{DavisKirk} for the fiber, for instance).

This implies that the stable homotopy groups $\pi_*(\Sinf (BG^{\wedge}_{p}))$ are all finite $p$-groups: This is a Serre class argument with the (convergent) Atiyah-Hirzebruch spectral sequence.

This implies that the spectrum $\Sinf (BG^{\wedge}_{p})$ is $p$-complete: As $\Sinf (BG^{\wedge}_{p})$ is connective, it suffices to prove that the spectrum is $H\F_p$-local. Let $X$ be an $H\F_p$-acyclic spectrum. Let $Y_i$ be the $i$th stage in the Postnikov tower for $\Sinf (BG^{\wedge}_{p})$ so that
\[
\Sinf (BG^{\wedge}_{p}) \simeq \lim Y_i
\]
and
\[
\Sinf (BG^{\wedge}_{p})^X \simeq \lim Y_{i}^{X}.
\]
We would like to show that this spectrum is zero. Let $K_i$ be the fiber of the map $Y_i \rightarrow Y_{i-1}$. By induction, it suffices to prove the $K_{i}^{X} \simeq 0$.

There is an equivalence $K_i \simeq \Sigma^k HA$ for a finite abelian $p$-group group $A$. Thus it suffices to show that $(\Sigma^k H\Z/p^l)^X \simeq 0$ for all $l$. By induction on the fiber sequence
\[
\Sigma^k H\F_p \rightarrow \Sigma^k H\Z/p^l \rightarrow \Sigma^k H\Z/p^{l-1}
\]
it suffices to prove that $(\Sigma^k H\F_p)^X \simeq 0$. This is the spectrum of $H\F_p$-module maps $\mathrm{Mod}_{H\F_p}(H\F_p \wedge X, \Sigma^k H\F_p)$. By assumption $H\F_p \wedge X \simeq *$.

This implies that the canonical map
\[
\Sinf BG \lra{} \Sinf (BG^{\wedge}_{p})
\]
factors through
\[
(\Sinf BG)^{\wedge}_{p} \lra{} \Sinf (BG^{\wedge}_{p})
\]
which is an $H\F_p$-homology equivalence between $H\F_p$-local spectra and thus is an equivalence.
\end{proof}

When restricted to the homotopy category of classifying spectra of finite groups, the $p$-completion functor has a simple description. We have not found this fact in the literature.

\begin{proposition} \label{algebraic}
The $p$-completion functor on spectra induces an isomorphism
\[
[\Sinfp BG, \Sinfp BH]^{\wedge}_{p} \lra{\cong} [\Sinfpp BG, \Sinfpp BH].
\]
\end{proposition}
\begin{proof}
We have splittings
\[
[\Sinfp BG, \Sinfp BH] \cong [S^0, S^0] \oplus [\Sinf BG, S^0] \oplus [\Sinfp BG, \Sinf BH].
\]
and
\[
\Sinf BG \simeq \bigvee_{l} (\Sinf BG)^{\wedge}_{l} \text{ and } \Sinf BH \simeq \bigvee_{l} (\Sinf BH)^{\wedge}_{l},
\]
where the wedges are over primes dividing the order of the group. Since
\[
[X, (\Sinf BH)^{\wedge}_{l}]
\]
is $l$-complete for finite type $X$ and a prime $l$ and thus algebraically $l$-complete (\cite{Bousfield2}*{Proposition 2.5}), the (algebraic) $p$-completion of the abelian group
\[
[\Sinfp BG, \Sinf BH]
\]
is $[\Sinfp BG, \Sinfh BH]$.
The algebraic $p$-completion of $[S^0,S^0]$ is clearly $\Z_p \cong [S^0, (S^0)^{\wedge}_{p}]$.

Finally, we must deal with $[\Sinf BG, S^0]$. However, since $\Sinf BG$ is connected, maps from $\Sinf BG$ to $S^0$ factor through the connected cover of $S^0$, $\tilde{S}^0$. The connected cover of $S^0$ has trivial rational cohomology, so the arithmetic fracture square gives a splitting
\[
\tilde{S}^0 \simeq \bigvee_{l} (\tilde{S}^0)^{\wedge}_{l} \simeq \Prod{l} (\tilde{S}^0)^{\wedge}_{l}.
\]
The second equivalence is a consequence of the fact that $\pi_i S^0$ is finite above degree zero. Thus the group $[\Sinf BG, S^0]$ appears to be an infinite product, however the $l$-completion of $\Sinf BG$ is contractible for $l$ not dividing the order of $G$. Thus the product
\[
\Prod{l} [\Sinf BG, (\tilde{S}^0)^{\wedge}_{l}]
\]
has a finite number of non-zero factors. The $p$-completion of this product is just the factor corresponding to the prime $p$.
\end{proof}

\section{A formula for the $p$-completion functor}\label{secAFormulaForPCompletion}
Fix a prime $p$. We give an explicit formula for the $p$-completion functor from virtual $(G,H)$-bisets to $p$-complete spectra sending a biset
\[
X \colon \Sinfp BG \rightarrow \Sinfp BH
\]
to the $p$-completion
\[
X^{\wedge}_{p} \colon \Sinfpp BG \rightarrow \Sinfpp BH.
\]

\subsection{Further results on Burnside modules for fusion systems}
We begin with a result describing how Burnside modules for fusion systems relate to the stable homotopy category. Let $G$ and $H$ be finite groups and let $S \subseteq G$ and $T \subseteq H$ be fixed Sylow $p$-subgroups. Let $\cF_G$ and $\cF_H$ be the fusion systems on $S$ and $T$ determined by $G$ and $H$. Recall that there are natural forgetful maps
\[
\AG(G,H) \rightarrow \AG(S,T)
\]
and
\[
\AF(\cF_G, \cF_H) \hookrightarrow \AF(S,T).
\]
In the stable homotopy category the analogous maps are the map
\[
[\Sinfp BG, \Sinfp BH] \rightarrow [\Sinfp BS, \Sinfp BT]
\]
given by composing with the inclusion from $\Sinfp BS$ and the transfer to $\Sinfp BT$ and the map
\[
[\Sinfpp B\cF_G, \Sinfpp B\cF_H] \rightarrow [\Sinfpp BS, \Sinfpp BT]
\]
given by precomposing with the restriction $r$ and postcomposing with the transfer $t$ (see Section \ref{secFusionSystems}). We may use this to construct a map
\[
[\Sinfp BG, \Sinfp BH] \rightarrow [\Sinfpp B\cF_G, \Sinfpp B\cF_H]
\]
by sending $X$ to the $p$-completion of the composite
\[
\Sinfp B\cF_G \lra{t} \Sinfp BS \lra{{}_{S}G_{G}} \Sinfp BG \lra{X} \Sinfp BH \lra{{}_{H}H_{T}} \Sinfp BT \lra{r} \Sinfp B\cF_H.
\]

\begin{proposition} \label{aniso}
Let $\cF_1$ and $\cF_2$ be saturated fusion systems on the $p$-groups $S_1$ and $S_2$. There is a canonical isomorphism
\[
\AF(\cF_1, \cF_2) \lra{\cong} [\Sinfpp B\cF_1, \Sinfpp B\cF_2].
\]
\end{proposition}
\begin{proof}
The abelian groups $\AF(S_1, S_2)$ and $[\Sinfpp BS_1, \Sinfpp BS_2]$ both have canonical idempotent endomorphisms given by precomposing and postcomposing with the characteristic idempotents associated to $\cF_1$ and $\cF_2$. The algebraic characteristic idempotent maps to the spectral characteristic idempotent by definition. This compatibility ensures that the images of these endomorphisms map to each other under the canonical isomorphism of Proposition \ref{algebraic}
\[
\AF(S_1, S_2) \lra{\cong} [\Sinfpp BS_1, \Sinfpp BS_2].
\]
The images of these endomorphisms are precisely the $(\cF_1, \cF_2)$-stable bisets and the homotopy classes
\[
[\Sinfpp B\cF_1, \Sinfpp B\cF_2].
\]
Since the retract of an isomorphism is an isomorphism, we have constructed a canonical isomorphism
\[
\AF(\cF_1, \cF_2) \lra{\cong} [\Sinfpp B\cF_1, \Sinfpp B\cF_2].
\qedhere\]
\end{proof}

\begin{proposition} \label{bigdiagram}
Let $G$ and $H$ be finite groups and let $S \subset G$ and $T \subset H$ be Sylow $p$-subgroups. Let $\cF_G$ and $\cF_H$ be the fusion systems on $S$ and $T$ determined by $G$ and $H$. There is a commutative diagram
\[
\xymatrix{\AG(G,H) \ar@/_4pc/[ddd] \ar[r] \ar[d] & [\Sinfp BG, \Sinfp BH] \ar[d] \ar@/^4pc/[ddd] \\ \AG(S,T) \ar[r] \ar[d]  & [\Sinfp BS, \Sinfp BT] \ar[d] \\ \AF(S,T) \ar[r]^-{\cong} & [\Sinfpp BS, \Sinfpp BT] \\ \AF(\cF_G, \cF_H) \ar[r]^-{\cong} \ar@{^{(}->}[u] & [\Sinfpp B\cF_G, \Sinfpp B\cF_H]. \ar@{^{(}->}[u]}
\]
\end{proposition}

\begin{proof}
The top center square commutes by naturality of Theorem \ref{SegalConj}. The middle center square commutes by the corollaries to Theorem \ref{SegalConj}. The bottom middle square commutes by the discussion in the proof of Proposition \ref{aniso}.

The left part of the diagram commutes as the restriction of a $(G,H)$-biset is bistable. Note that the map
\[
[\Sinfpp B\cF_G, \Sinfpp B\cF_H] \rightarrow [\Sinfpp BS, \Sinfpp BT]
\]
is given by the formula
\[
f \mapsto tfr.
\]
The right part of the diagram commutes as the $p$-completion of the composite
\[
\Sinfp B\cF_G \lra{t} \Sinfp BS \lra{{}_{S}G_{G}} \Sinfp BG \lra{X} \Sinfp BH \lra{{}_{H}H_{T}} \Sinfp BT \lra{r} \Sinfp B\cF_H
\]
maps to the $p$-completion of the composite
\[
\Sinfp BS \lra{\omega_{\cF_G}} \Sinfp BS \lra{{}_{S}G_{G}} \Sinfp BG \lra{X} \Sinfp BH \lra{{}_{H}H_{T}} \Sinfp BT \lra{\omega_{\cF_H}} \Sinfp BT
\]
and stability implies that this is equal to the $p$-completion of
\[
\Sinfp BS \lra{{}_{S}G_{G}} \Sinfp BG \lra{X} \Sinfp BH \lra{{}_{H}H_{T}} \Sinfp BT.
\qedhere\]
\end{proof}

Proposition \ref{bigdiagram} gives an interpretation of the biset construction
\[
{}_{G}X_{H} \mapsto {}_{\cF_G}X_{\cF_H}
\]
in terms of spectra. It is the $p$-completion of the composite
\[
\Sinfp B\cF_G \lra{t} \Sinfp BS \lra{{}_{S}G_{G}} \Sinfp BG \lra{X} \Sinfp BH \lra{{}_{H}H_{T}} \Sinfp BT \lra{r} \Sinfp B\cF_H.
\]

Now we focus on the relationship between $\Sinfpp B\cF_G$ and $\Sinfpp BG$. Applying $p$-completion to the maps $t \colon \Sinfp B\cF_G \rightarrow \Sinfp BS$ and ${}_{S}G_{G} \colon \Sinfp BS \rightarrow \Sinfp BG$ gives us the commutative diagram
\begin{equation} \label{aGdiagram}
\xymatrix{\Sinfp B\cF_G \ar[r]^{t} \ar[d] & \Sinfp BS \ar[r]^{{}_{S}G_{G}} \ar[d] & \Sinfp BG \ar[d] \\ \Sinfpp B\cF_G \ar[r] \ar@/_1pc/[rr]_{a_G}  & \Sinfpp BS  \ar[r] & \Sinfpp BG.}
\end{equation}
The map $a_G$ is the composite of the bottom arrows.

\begin{proposition} \label{aG} (Essentially \cite{CartanEilenberg}*{XII.10.1} and \cite{BLO2}*{Proposition 5.5})
The map
\[
a_G \colon \Sinfpp B\cF_G \rightarrow \Sinfpp BG
\]
is an equivalence.
\end{proposition}

\begin{proof}[Proof] We provide a proof for completeness.
By \cite{CartanEilenberg}*{XII.10.1} the inclusion of subgroups ${}_S G_G \colon \Sinfp BS \rightarrow \Sinfp BG$ induces an inclusion in mod-$p$ cohomology $H^*(BG;\F_p) \to H^*(BS;\F_p)$ with image the $\cF_G$-stable elements of $H^*(BS;\F_p)$.

Simultaneously $\Sinfp B\cF_G$ is constructed as the image in $\Sinfp BS$ when applying the idempotent $\omega_\cF$, and by \cite{BLO2}*{Proposition 5.5} in cohomology the idempotent $\omega_\cF$ induces a projection of $H^*(BS;\F_p)$ onto the subring of $\cF$-stable elements.

Consequently, the map $f\colon \Sinfp B\cF \to \Sinfp BS\to \Sinfp BG$ first includes $H^*(BG;\F_p)$ as the $\cF$-stable elements of $H^*(BS;\F_p)$ which is then projected, by the identity on $\cF$-stable elements, onto $H^*(\Sinfp B\cF;\F_p)$. Hence $f$ is a mod-$p$ equivalence, and the $p$-completion of $f$, $a_G$, is an equivalence.
\end{proof}

Note that $a_G$ is canonical and natural in maps of finite groups. We will use this equivalence to identify these spectra with each other.

Let $S \subset G$ be a Sylow $p$-subgroup and define
\[
J_G = \ker(A(G) \rightarrow A(S)).
\]
Note that $J_G$ is independent of the choice of Sylow $p$-subgroup. We provide a purely algebraic proof of the next folklore proposition inspired by the proof of \cite[Lemma 5]{MayMcClure}. It is possible to give a shorter proof by making use of the Segal conjecture following the lines of \cite[Proposition 9.7]{Strickland}, but we find the algebraic proof quite satisfying.

\begin{proposition}\label{propIsoOfCompletions}
Let $G$ be a finite group and let $J_G \subset A(G)$ be the ideal defined above. The are isomorphisms
\[
A(G)^{\wedge}_{p+I_G} \cong \Z_p \otimes A(G)/J_G \cong A(\cF_G)^\wedge_p
\]
natural in $G$.
\end{proposition}

\begin{proof}
Write the order of $G$ as $\abs{G}=p^k\cdot u$ where $u$ is the index of $S$ in $G$ and coprime to $p$.
We many times throughout the proof make use of the classical B{\'e}zout identity that we can write $1$ as an integral linear combination
\begin{equation}\label{eqBezoutFormula}
1 = a\cdot u + m\cdot p.
\end{equation}
For the first part of the proof, we show $A(G)^{\wedge}_{p+I_G} \cong \Z_p \otimes A(G)/J_G$ by constructing canonical maps in both directions.

For a virtual $G$-set $X\in A(G)$, the product $G/S \times X$ is isomorphic to $G\times_S X$ which takes the restriction of $X$ to $S$ and induces back up to $G$. Since $J_G=\ker(A(G)\xrightarrow{} A(S))$, we therefore conclude that
\[(G/S) J_G = 0.\]
The virtual $G$-set $u\cdot(G/G) - (G/S)$ is an element of the augmentation ideal $I_G$, and by the preceding calculation we have
\[\bigl(u\cdot(G/G) - (G/S) \bigr) J_G = u J_G.\]
This combined with \eqref{eqBezoutFormula} easily proves that $J_G\subseteq ((p)+I_G)J_G$:
\begin{multline*}
J_G = (mp+au)J_G \subseteq pJ_G+ uJ_G = pJ_G + \bigl(u(G/G) - (G/S)\bigr)J_G \\\subseteq pJ_G + I_GJ_G = ((p)+I_G)J_G.
\end{multline*}
By iterating the equation above, it follows that $J_G$ is contained in all powers of $(p)+I_G$, so $J_G$ is contained in the kernel of the completion map $A(G)\to A(G)^\wedge_{p+I_G}$.

Consequently we have a well-defined map
\[A(G)/J_G \to A(G)^\wedge_{p+I_G},\]
and since $A(G)^\wedge_{p+I_G}$ is in particular $p$-complete, the map extends to the $p$-completion $\Z_p\otimes (A(G)/J_G)$.

The map in the other direction requires a little more work. Recall that $A(G)$ embeds into a product ring (the so-called ghost ring $\Omega_G$):
\[
\Phi\colon A(G) \to \prod_{\substack{H\leq G \\ \text{up to $G$-conj.}}} \Z
\]
where the $H$-coordinate counts the $H$-fixed points, $\Phi_H(X) = \abs{X^H}$. The cokernel of $\Phi$ is finite, isomorphic to $\Omega_G/\Phi(A(G)) \cong \prod_H \Z/\abs{N_G H / H}\Z$, hence the ghost ring $\Omega_G$ satisfies that
\[\abs{G}\cdot \Omega_G \subseteq \Phi(A(G)).\]
Let $I\Omega_G$ be the augmentation ideal of $\Omega_G$, consisting of all tuples where the coordinate at the trivial subgroup equals zero.

Consider the transitive $G$-set $G/S$ and its restriction ${}_S(G/S)\in A(S)$. For each subgroup $P\leq S$, the fixed points for $P$ satisfy $p\nmid \abs{(G/S)^P}$. It follows that $\Phi({}_S(G/S))$ is invertible in the $p$-localization $(\Omega_S)_{(p)}$ of the ghost ring.
Multiplication by $\Phi({}_S (G/S))$ gives an automorphism of $(\Omega_S)_{(p)}$ that takes $\Phi(A(S)_{(p)})$ to some subset of itself. Since the cokernel of $\Phi\colon A(S)_{(p)}\to (\Omega_S)_{(p)}$ is finite and $\Phi({}_S(G/S))$ takes the image to itself, multiplication by ${}_S(G/S)$ must be an automorphism of $A(S)_{(p)}$ as well.

The inverse ${}_S(G/S)^{-1} \in A(S)_{(p)}$ has coefficients in $\Z_{(p)}$ so there exists some positive integer $v$, not divisible by $p$, such that
\[v\cdot {}_S(G/S)^{-1} \in A(S).\]
Because ${}_S(G/S)$ is $\cF_G$-stable, the inverse ${}_S(G/S)^{-1}$ is $\cF$-stable as well.

We now define a virtual $G$-set $M$ by induction of the virtual $S$-set above:
\[M := G\times_S (v\cdot {}_S(G/S)^{-1})\in A(G).\]
All orbits in $M$ has $p$-group stabilizers, so $\abs{M^H}=0$ for all non-$p$-subgroups $H\leq G$. Furthermore, the restriction of $M$ back to $S$ becomes
\[{}_S M = {}_SG \times_S (v\cdot {}_S(G/S)^{-1}).\]
The biset ${}_S G_S$ decomposes into orbits according to the double cosets
\[{}_S G_S \cong \sum_{x\in S\backslash G/S} [S\cap xSx^{-1}, c_x]_S^S \text{ in $A(S,S)$}.\]
Because $v\cdot {}_S(G/S)^{-1}$ is $\cF_G$-stable, restricting along a conjugation map $c_x\colon S\cap xSx^{-1}\to S$ is isomorphic to just restricting along the inclusion $S\cap xSx^{-1}\hookrightarrow S$. Consequently, acting by the biset ${}_S G_S$ on $v\cdot {}_S(G/S)^{-1}$ is equivalent to just multiplying with the $S$-set
\[{}_S (G/S) \cong \sum_{x\in S\backslash G/S} S/(S\cap xSx^{-1}).\]
We therefore have
\[{}_S M =  {}_SG \times_S (v\cdot {}_S(G/S)^{-1}) \cong {}_S(G/S) \times \bigl(v\cdot {}_S(G/S)^{-1}\bigr) = v\cdot (S/S).\]
From $M$ we can now construct the element $(v\cdot (G/G) - M)\in J_G$ since the summands cancel each other on restriction to $S$.
For every $X\in A(G)$ we have
\[v\cdot X = (v\cdot (G/G) - M)\times X + M\times X.\]
The first summand $(v\cdot (G/G) - M)\times X$ lies in the ideal $J_G$, and the second summand $M\times X$ has only $p$-group stabilizers and trivial fixed points for all non-$p$-subgroups $H\leq G$.

Recall that the order of $S$ is $p^k$. We shall prove that $I_G^{k+1} \subseteq J_G + pI_G$ in analogy to Lemma \ref{lemmaIpadic}, and this will allow us to define a map $A(G)^\wedge_{p+I_G} \to \Z_p\otimes (A(G)/J_G)$.
Let $X$ be any element of $I_G$, then $M\times X$ still has augmentation $0$. The virtual $G$-set $M\times X$ has $\abs{(M\times X)^H}=0$ for non-$p$-subgroups $H\leq G$ and additionally $\abs{M\times X}=0$. For $p$-subgroups $P\leq G$, we always have $p\mid \abs{Y} - \abs{Y^P}$ for $Y\in A(G)$, so in particular $p\mid \abs{(M\times X)^P}$ for all non-trivial $p$-subgroups $P\leq G$. We conclude that $p$ divides all coordinates of $\Phi(M\times X)\in \Omega_G$. Hence we have $M\times X\in pI\Omega_G$, and for $I_G$ in general we can write
\[v\cdot \Phi(I_G) \subseteq \Phi\bigl((v\cdot (G/G)-M)I_G\bigr) + \Phi(M\cdot I_G) \subseteq \Phi(J_G) + (pI\Omega_G \cap \Phi(I_G)).\]
Recall the earlier B{\'e}zout formula \eqref{eqBezoutFormula} for $u$ and recall that $p^k\cdot u \Omega_G \subseteq \Phi(A(G))$, which also holds for the augmentation ideals. With these results in mind, we see that
\begin{align*}
v^{k+1} \Phi(I_G^{k+1}) &\subseteq \Phi(J_G) + (p^{k+1} I\Omega_G \cap \Phi(I_G))
\\ &\subseteq \Phi(J_G) + p^{k+1}\cdot u\cdot I\Omega_G + p \Phi(I_G)
\\ &\subseteq \Phi(J_G) + p \Phi(I_G) + p\Phi(I_G)
\\ &=  \Phi(J_G) + p \Phi(I_G).
\end{align*}
To get rid of the $v$'s, we can use another B{\'e}zout identity, $1= b\cdot v + n\cdot p$, and see that
\[I_G^{k+1} = (b\cdot v+ n\cdot p)^{k+1} I_G^{k+1} \subseteq v^{k+1} I_G^{k+1} + p I_G \subseteq J_G + p I_G + pI_G = J_G + p I_G.\]
From this formula we see that
\[((p)+I_G)^{k+1} \subseteq J_G + p A(G).\]
By taking repeated powers we conclude that there is a well-defined map
\[A(G)^\wedge_{p+I_G} \to \Z_p\otimes (A(G)/J_G).\]
The maps between $A(G)^\wedge_{p+I_G}$ and $\Z_p\otimes (A(G)/J_G)$ in both directions are given by taking sequences of representatives in $A(G)$ and taking the limit in the other ring, hence the two maps are inverse to each other, and $A(G)^\wedge_{p+I_G}\cong\Z_p\otimes (A(G)/J_G)$ as claimed.

The second isomorphism $\Z_p\otimes (A(G)/J_G)\cong A(\cF_G)^\wedge_p$ is easier to see. The ideal $J_G$ is the kernel of the restriction $A(G)\to A(S)$, hence $A(G)/J_G$ is isomorphic to the image of $A(G)\to A(S)$. The image of $A(G)\to A(S)$ is contained in the subring $A(\cF_G)$ of $\cF_G$-stable elements. By Proposition 4.12 of \cite{ReehIdempotent}, $p$-locally the restrictions of $G$-sets generate all of $A(\cF_G)_{(p)}$ with a $p$-local basis consisting of the elements 
\[\frac{{}_S (G/P)}{{}_S (G/S)} \in A(\cF_G)_{(p)},\]
for $P\leq S$ and where these basis elements only depend on $\cF_G$ instead of the entire group $G$. By tensoring with $\Z_p$ these fractions also form a $\Z_p$-basis for $A(\cF_G)^\wedge_p$, so $\Z_p\otimes (A(G)/J_G) \cong A(\cF_G)^\wedge_p$.
\end{proof} 

We draw the following folklore corollary:

\begin{corollary}
Let $G$ and $H$ be finite groups and let $J_G \subset A(G)$ be the ideal defined above. There are canonical isomorphisms
\[
\AG(G,H)^{\wedge}_{p+I_G} \cong \Z_p \otimes \AG(G,H)/J_G\AG(G,H) \cong \AF(\cF_G,\cF_H)
\]
natural in $G$ and $H$.
\end{corollary}
\begin{proof}
Since $A(G)^{\wedge}_{p+I_G} \cong \Z_p \otimes A(G)/J_G$ by Proposition \ref{propIsoOfCompletions}, the first isomorphism follows from the fact that the completion is given by base change. Naturality of this isomorphism follows from the fact that the image of a Sylow $p$-subgroup under a group homomorphism is contained in a Sylow $p$-subgroup.

To see the other isomorphism, recall that Theorem \ref{SegalConj} gives an isomorphism
\[
\AG(G,H)^{\wedge}_{I_G} \lra{\cong} [\Sinfp BG, \Sinfp BH].
\]
In view of Proposition \ref{aG}, it suffices to show that the $p$-completion functor
\[
[\Sinfp BG, \Sinfp BH] \lra{} [\Sinfpp BG, \Sinfpp BH]
\]
is given algebraically by $p$-completion. This is Proposition \ref{algebraic}.
\end{proof}

\subsection{A formula for $p$-completion}

Now consider the map ${}_{S}G_{S} \colon \Sinfp BS \rightarrow \Sinfp BS$, which is the composite of the inclusion and transfer along $S \subseteq G$. This element is $\cF_G$-semicharacteristic; it is not $S$-semicharacteristic as it is the composite
\[
{}_{S}G_{S} = [S, \id_S]_{S}^{G} \times_G [S,i_S]_{G}^{S}
\]
and various conjugations with respect to elements of $G$ not in $S$ show up in the resulting sum.

\begin{lemma} \label{invertiblemap}
The element
\[
{}_{\cF_G}G_{\cF_G} \in \AF(\cF_G, \cF_G) \cong [\Sinfpp B\cF_G, \Sinfpp B\cF_G]
\]
is a unit.
\end{lemma}
\begin{proof}
Since ${}_{\cF_G}G_{\cF_G}$ is $\cF_G$-semicharacteristic it is in the image of the inclusion
\[
i \colon A_{p}^{\text{char}}(\cF_G) \rightarrow \AF(\cF_G,\cF_G).
\]

Recall that the composite
\[
A_{p}^{\text{char}}(\cF_G) \rightarrow \AF(\cF_G,\cF_G) \rightarrow A(\cF_G)^{\wedge}_{p}
\]
is an isomorphism of commutative rings even though the second map is not a ring map. The image of ${}_{\cF_G}G_{\cF_G}$ in $A(\cF_G)^{\wedge}_{p}$ is $G/S$ viewed as an $\cF_G$-stable set. Since $|G/S|$ is coprime to $p$, this projects onto a unit in $\F_p$ under the canonical map
\[
A(\cF_G)^{\wedge}_{p} \rightarrow \Z_p \rightarrow \F_p.
\]
Since $A(\cF_G)^{\wedge}_{p}$ is complete local with maximal ideal $p+I_{\cF_G}$ (see Remark \ref{AFLocal}), $G/S$ is a unit, but now this implies that ${}_{\cF_G}G_{\cF_G}$ is a unit.
\end{proof}

Recall the equivalence of Proposition \ref{aG}, we now give an explicit description of the inverse to $a_G$. Consider the following diagram
\begin{equation} \label{inverse}
\xymatrix{\Sinfp BG \ar[r]^{{}_{G}G_{S}} \ar[d] & \Sinfp BS \ar[d] \ar[r]^{r} & \Sinfp B\cF_G \ar[d] \\ \Sinfpp BG  \ar[r] \ar@/_1pc/[rrd]_{b_G} & \Sinfpp BS \ar[r] & \Sinfpp B\cF_G \ar[d]_{\simeq}^{({}_{\cF_G}G_{\cF_G})^{-1}} & \\ && \Sinfpp B\cF_G.}
\end{equation}
The map ${}_{G}G_{S}$ is the transfer from $\Sinfp BG$ to $\Sinfp BS$. The second row is the $p$-completion of the first row. The map $({}_{\cF_G}G_{\cF_G})^{-1}$ exists by Lemma \ref{invertiblemap}, which depends on the fact that we are in the category of $p$-complete spectra. The map $b_G$ is the composite.

\begin{lemma}
The map
\[
b_G \colon \Sinfpp BG \rightarrow \Sinfpp B\cF_G.
\]
is the inverse to $a_G$.
\end{lemma}
\begin{proof}
Put Diagram \ref{aGdiagram} from page \pageref{aGdiagram} to the left of Diagram \ref{inverse}. Note that the image of ${}_{G}G_{G}$ along the map
\[
\AG(G,G) \rightarrow \AF(\cF_G,\cF_G)
\]
of Proposition \ref{bigdiagram} is ${}_{\cF_G}G_{\cF_G}$, which we then postcompose with $({}_{\cF_G}G_{\cF_G})^{-1}$.
\end{proof}

Using $b_G$, we may replace the target of the canonical map to the $p$-completion
\[
\Sinfp BG \lra{} \Sinfpp BG
\]
by $\Sinfpp B\cF_G$. Let
\[
c_G \colon \Sinfp BG \rightarrow \Sinfpp BG \lra{b_G} \Sinfpp B\cF_G
\]
be the composite; it is naturally equivalent to the $p$-completion map. Diagram \ref{inverse} gives a kind of formula for $c_G$. It is the the transfer ${}_{G}G_{S}$, viewed as a map landing in $\Sinfp B\cF_G$, postcomposed with the $p$-completion map $\Sinfp B\cF_G \rightarrow \Sinfpp B\cF_G$ followed by the equivalence $({}_{\cF_G}G_{\cF_G})^{-1}$:
\begin{equation}\label{cGformula}
c_G\colon \Sinfp BG \xrightarrow{{}_GG_{S}} \Sinfp BS \lra{r} \Sinfp B\cF_G \rightarrow \Sinfpp B\cF_G \xrightarrow{{({}_{\cF_G}G_{\cF_G})^{-1}}} \Sinfpp B\cF_G.
\end{equation}

Using the equivalences $a_G$ and $b_H$, we may construct a map
\[
\widehat{(-)} \colon [\Sinfp BG, \Sinfp BH] \rightarrow [\Sinfpp B\cF_G, \Sinfpp B\cF_H]
\]
by sending
\[
f \colon \Sinfp BG \rightarrow \Sinfp BH
\]
to the composite
\[
\widehat{f} \colon \Sinfpp B\cF_G \lra{a_G} \Sinfpp BG \lra{f^{\wedge}_{p}} \Sinfpp BH \lra{b_H} \Sinfpp B\cF_H.
\]
By Proposition \ref{aniso}, this is an element in $\AF(\cF_G,\cF_H)$. We will give a simple formula for $\widehat{f}$ when $f$ comes from a virtual $(G,H)$-biset. In fact, we may precompose $\widehat{(-)}$ with the canonical map $\AG(G,H) \rightarrow [\Sinfp BG, \Sinfp BH]$. By abuse of notation, we will also call this $\widehat{(-)}$.


\begin{theorem} \label{mainthm}
Let $G$ and $H$ be finite groups, and let $T \subset H$ be the Sylow $p$-subgroup on which $\cF_H$ is defined. The map
\[
\AG(G,H) \lra{\widehat{(-)}} \AF(\cF_G,\cF_H)
\]
sends a virtual $(G,H)$-biset ${}_{G}X_{H}$ to
\[
{}_{\cF_G}\widehat{X}_{\cF_H} = {}_{\cF_G}X \times_T H^{-1}_{\cF_H} =  ({}_{\cF_G}X_{\cF_H}) \times_{T} ({}_{\cF_H}H_{\cF_H})^{-1}.
\]
In other words, there is a commutative diagram in the stable homotopy category
\[
\xymatrix{\Sinfp BG \ar[rrr]^X \ar[d]_{c_G} & & & \Sinfp BH \ar[d]^{c_H} \\ \Sinfpp B\cF_G \ar[rrr]^{{}_{\cF_G}X \times_T H^{-1}_{\cF_H}} & & & \Sinfpp B\cF_H.}
\]
\end{theorem}
\begin{proof}
Putting together Diagram \ref{aGdiagram} and the definition of $c_H$ as $\Sinfp BH \to \Sinfpp BH \xrightarrow{b_H} \Sinfpp B\cF_H$, we have a commutative diagram
\[
\xymatrix{\Sinfp B\cF_G \ar[r] \ar[d] & \Sinfp BS \ar[r] & \Sinfp BG \ar[d] \ar[r]^{X} & \Sinfp BH \ar[dr]^{c_H} \ar[d] & \\ \Sinfpp B\cF_G \ar[rr]^{a_G} && \Sinfpp BG \ar[r]^{X^{\wedge}_p} & \Sinfpp BH \ar[r]^{b_H} & \Sinfpp B\cF_H.}
\]
The vertical arrows are all the canonical maps to the $p$-completion. Note that $\widehat{X}$ is the is the composite of the arrows in the bottom row of the diagram. Also, we could add $c_G \colon \Sinfp BG \to \Sinfpp BG \xrightarrow{a_G^{-1}} \Sinfpp B\cF_G$ diagonally in the left hand square and the diagram would still commute. Plug in the expression \eqref{cGformula} for $c_G$, and the composite along the top of the diagram is precisely
\[
{}_{\cF_G}X \times_T H^{-1}_{\cF_H}
\]
giving us the desired formula.
\end{proof}

\begin{remark} The formula of Theorem \ref{mainthm} also makes sense on elements in $\AG(G,H)^{\wedge}_{I_G}$ and the proof is identical once one feels comfortable referring to elements in the completion as virtual bisets.
\end{remark}

This result allows us to give explicit formulas for the $p$-completion functor. It is often useful to have formulas more explicit than $({}_{\cF_G}G_{\cF_G})^{-1}$. We will give two further ways of understanding this element. One as an infinite series and the other as a certain limit. The following formulas for calculating $({}_{\cF_G}G_{\cF_G})^{-1}$ are based on similar calculations in \cite{Ragnarsson}.


\begin{proposition}\label{inverseFormula}
Let $X = {}_{\cF_G}G_{\cF_G}$. Inside $\AF(\cF_G,\cF_G)$ we have the equalities
\[
X^{-1} = X^{p-2}\sum_{i \geq 0} (1-X^{p-1})^{i} = \lim_{n \to \infty} X^{(p-1)p^n-1}.
\]
\end{proposition}
\begin{proof}
Since $X$ is $\cF_G$-characteristic, it suffices to prove this inside
\[
A(\cF_G)^{\wedge}_{p} \cong A^{\text{char}}_{p}(\cF_G) \subset \AF(\cF_G,\cF_G).
\]
This ring is complete local with maximal ideal $m = p + I_{\cF_G}$ (see Remark \ref{AFLocal}).

Since $X$ is invertible, it is enough to show that
\[
\lim_{n \to \infty} X^{(p-1)p^n}= X \times \lim_{n \to \infty} X^{(p-1)p^n-1} = 1.
\]
But since $A(\cF_G)^{\wedge}_{p}/m \cong \F_p$, we have that $X^{p-1} = 1 \mod m$. Thus it is enough to show that if $Y=1 \mod m$, then $\lim\limits_{n \to \infty} Y^{p^n}=1$.

Indeed, we will prove by induction that
\[
Y^{p^n}-1 \in m^{n+1}.
\]
The case $n=0$ follows by assumption. Assume that $Y^{p^n}-1 \in m^{n+1}$. We have
\[
Y^{p^{n+1}}-1  = (Y^{p^n}-1)(\sum^{p-1}_{i=0} Y^{p^ni}),
\]
but
\[
\sum^{p-1}_{i=0} Y^{p^ni} = \sum^{p-1}_{i=0} 1^{p^ni} =p =0 \mod m
\]
and $(Y^{p^n}-1)\in m^{n+1}$ by assumption, so $(Y^{p^{n+1}}-1)\in m^{n+2}$.

For the other equality, note that the sum is the geometric series for $1/X^{p-1}$ and that the summand live in higher and higher powers of $m$.
\end{proof}

\begin{remark}
The formulas for $X^{-1}$ in the previous proposition are true much more generally. For instance, it suffices that $R$ is a (not necessarily commutative) $\Z_p$-algebra with a two-sided maximal ideal $m$ such that $R/m \cong \F_p$ and $R$ is finitely generated as a $\Z_p$-module.
\end{remark}

\begin{corollary}\label{splittingIdempotent}
The idempotent in $\KG(G,G)^\wedge_{I_G}$ that splits off $(\Sinf BG)^\wedge_p$ as a summand of $\Sinf BG$ can be written as
\[
\lim_{n\to\infty} ([S,i_S]_G^G - [S,0]_G^G)^{(p-1)p^n}.
\]
The biset $[S,i_S]_G^G - [S,0]_G^G= (G\times_S G) - (G/S \times_{e} G)$ is just the transfer from $\Sinf BG$ to $\Sinf BS$ followed by the inclusion back to $\Sinf BG$.
\end{corollary}

\begin{proof}
Recall from Section \ref{remarkDisjointBasepoints} that the idempotent $[H,i_H] - [H,0]\in \AG(H,H)$ splits off $\Sinf BH$ as a summand of $\Sinfp BH$. Furthermore, for any virtual $(G,H)$-biset $X$, we have
\begin{equation}\label{eqRightsideIdempotent}
([G,i_G]-[G,0])\times_G X \times_H ([H,i_H]-[H,0]) = X \times_H ([H,i_H]-[H,0]).
\end{equation}
Hence in order to get the unpointed part $\Sinf BG\to \Sinf BH$ of any map $\Sinfp BG\to \Sinfp BH$, we just have to postcompose with $[H,i_H] - [H,0]\in \AG(H,H)$.

For a saturated fusion system $\cF$, the characteristic idempotent $\omega_\cF$, which splits $\Sinfpp BF$ from $\Sinfpp BS$, acts as the identity on the $(S^0)^\wedge_p$-summand since $\abs{\omega_\cF}=1$. Splitting off $\Sinf B\cF$ from $\Sinfpp B\cF$ corresponds to the idempotent
$\omega_\cF-[S,0]_S^S = \omega_\cF \times_S ([S,i_S]_S^S - [S,0]_S^S)\in \AF(\cF,\cF)$.

The map $c_G\colon \Sinfp BG\to \Sinfpp B\cF_G$ has an unpointed part $\overline c_G\colon \Sinf BG\to \Sinf B\cF_G$ which we get by postcomposing with the idempotent $\omega_{\cF_G}-[S,0]_S^S\in \AF(\cF,\cF)$. The map $\overline c_G$ is represented by the composition
\[
\begin{tikzpicture}
  \matrix (M) [matrix of math nodes] {
    \Sinfp BG &[1cm] \Sinfp BS &[1cm] \Sinfpp B\cF_G &[3cm] \Sinfpp B\cF_G \\[1cm]
    \Sinf BG &&& \Sinf B\cF_G \\
  };
  \path[auto,arrow,->]
    (M-1-1) edge node{${}_GG_{S}$} (M-1-2)
    (M-1-2) edge node{$r$} (M-1-3)
    (M-1-3) edge node{${({}_{\cF_G}G_{\cF_G})^{-1}}$} (M-1-4)
    (M-1-4) edge (M-2-4)
    (M-2-1) edge (M-1-1)
            edge node{${}_GG_{S} \times_S ({}_{\cF_G}G_{\cF_G})^{-1} \times_S (\omega_{\cF_G} - [S,0])$} (M-2-4)
  ;
\end{tikzpicture}
\]

Similarly the map $s\colon \Sinf B\cF_G \xrightarrow{t} \Sinf BS \xrightarrow{i_S} \Sinf BG$ is represented by the virtual biset ${}_S G_G \times_G ([G,i_G] - [G,0])$.

We see that $s$ is a section to $\overline c_G$ since $\overline c_G\circ s$ is represented by the composite
\begin{align*}
& {}_S G_G \times_G ([G,i_G] - [G,0]) \times_G {}_GG_{S} \times_S ({}_{\cF_G}G_{\cF_G})^{-1} \times_S (\omega_{\cF_G} - [S,0])
\\ ={}& {}_S G_G \times_G {}_GG_{S} \times_S ({}_{\cF_G}G_{\cF_G})^{-1} \times_S (\omega_{\cF_G} - [S,0])
\\ ={}& {}_S G_{S} \times_S ({}_{\cF_G}G_{\cF_G})^{-1} \times_S (\omega_{\cF_G} - [S,0])
\\ ={}& \omega_{\cF_G}\times_S (\omega_{\cF_G} - [S,0]) = (\omega_{\cF_G} - [S,0])
\end{align*}
and $\omega_{\cF_G} - [S,0]$ is the identity map on $\Sinf B\cF_G$. Note that we can leave out $([G,i_G] - [G,0])$ in the middle since we multiply by $(\omega_{\cF_G} - [S,0])$ at the end anyway.

The composition $\overline c_G\circ s$ ends with the equivalence $\overline b_G\colon (\Sinf BG)^\wedge_p \to \Sinf B\cF_G$. If we instead place $\overline b_G$ at the beginning, we see that
\[
(\Sinf BG^\wedge_p) \xrightarrow{b_G} (\Sinf B\cF_G) \xrightarrow{s} \Sinf BG \to (\Sinf BG)^\wedge_p
\]
is also the identity.

From this we conclude that $s\circ \overline c_G\colon \Sinf BG \to \Sinf BG$ is an idempotent whose image is equivalent to the $p$-completion $(\Sinf BG)^\wedge_p$. We now plug in the limit formula for $({}_{\cF_G}G_{\cF_G})^{-1}$ from Proposition \ref{inverseFormula} and see that the idempotent $s\circ \overline c_G$ has the form
\begin{align*}
& {}_GG_{S} \times_S ({}_{\cF_G}G_{\cF_G})^{-1} \times_S (\omega_{\cF_G} - [S,0]) \times_S {}_S G_G \times_G ([G,i_G] - [G,0])
\\ ={}& {}_GG_{S} \times_S ({}_{\cF_G}G_{\cF_G})^{-1} \times_S {}_S G_G \times_G ([G,i_G] - [G,0])
\\ ={}& {}_GG_{S} \times_S \Bigl(\lim_{n\to\infty} ({}_S G_S)^{(p-1)p^n-1} \Bigr) \times_S {}_S G_G \times_G ([G,i_G] - [G,0])
\end{align*}
Each factor ${}_S G_S$ in the limit of powers can be decomposed as $({}_S G_G) \times_G ({}_G G_S)$. Note that we have an additional ${}_G G_S$ in front of the limit and ${}_S G_G$ after the limit. We pull in the additional factors and make powers of $({}_G G_S)\times_S ({}_S G_G)=G\times_S G$ instead -- with the exponent increased by $1$:
\begin{align*}
& {}_GG_{S} \times_S \Bigl(\lim_{n\to\infty} ({}_S G_S)^{(p-1)p^n-1} \Bigr) \times_S {}_S G_G \times_G ([G,i_G] - [G,0])
\\ ={}& \Bigl(\lim_{n\to\infty} (G \times_S G)^{(p-1)p^n} \Bigr) \times_G ([G,i_G] - [G,0])
\\ ={}& \lim_{n\to\infty} (G \times_S G - [S,0]_G^G)^{(p-1)p^n}.
\end{align*}
The last equality holds because \eqref{eqRightsideIdempotent} tells us that we can act with the idempotent $[G,i_G] - [G,0]$  on every factor in a long composition, and $(G \times_S G)\times_G ([G,i_G] - [G,0]) = G \times_S G - [S,0]_G^G$.
\end{proof}

\begin{example}
As an example of how the different formulas of this paper play together with the $I_G$-adic topology, we will perform a ``sanity check''. We will check that the idempotents of Corollary \ref{splittingIdempotent} at each prime actually add up to give back the identity on $\Sinf BG$.

Let $S_p$ be a Sylow $p$-subgroup of $G$, and let $\omega_p$ denote the idempotent
\[
\omega_p:= \lim_{n\to\infty} ([S_p,i_{S_p}]_G^G - [S_p,0]_G^G)^{(p-1)p^n}
\]
that splits off $(\Sinf BG)^\wedge_p$ from $\Sinf BG$. We will confirm that
\[ [G,i_G] - [G,0] = \sum_p \omega_p\] in the endomorphism ring $\KG(G,G)^\wedge_{I_G}$ of $\Sinf BG$.

First note that we may write
\begin{equation} \label{equation}
1 = \sum_p a_p \frac{\lvert G\rvert}{\lvert S_p\rvert},
\end{equation}
a linear combination of integers for some choice of integers $a_p$. Next let
\begin{align*}
Z :=& \sum_p a_p \Bigl( \frac{\lvert G\rvert}{\lvert S_p\rvert} ([G,i_G] - [G,0]) - ([S_p,i_{S_p}]-[S_p,0]) \Bigr)
\\ ={}& \sum_p a_p \Bigl( \frac{\lvert G\rvert}{\lvert S_p\rvert} [G,i_G] - [S_p,i_{S_p}] \Bigr)\times_G ([G,i_g]-[G,0]),
\end{align*}
which is an element of $I_G\cdot([G,i_G]-[G,0])$.

We now claim that
\begin{equation}\label{eqProductInAugmentationIdeal}
Z \times_G ( ([G,i_G]-[G,0]) - \sum_p \omega_p ) = ([G,i_G]-[G,0]) - \sum_p \omega_p.
\end{equation}
To show this we need the following two calculations:
The first is the fact that
\[([S_p,i_{S_p}] - [S_p,0]) \times_G ([S_q,i_{S_q}] - [S_q,0]) = 0\text{ whenever $p \neq q$.}\]
This is because the double coset formula for the composition of these bisets contains only subgroups of the form $(S_p)^g\cap S_q=1$ for elements $g\in G$, and contributions from $[S_q,i_{S_q}]$ and $[S_q,0]$ cancel each other when restricted to the trivial subgroup.

Consequently, $([S_p,i_{S_p}] - [S_p,0]) \times_G \omega_q = 0$ as $\omega_q$ is formed by iterating $[S_q,i_{S_q}] - [S_q,0]$.

The second calculation we need is that
\[([S_p,i_{S_p}] - [S_p,0] ) \times_G \omega_p = [S_p,i_{S_p}] - [S_p,0],\]
which we can prove by reversing the last step in the proof of Corollary \ref{splittingIdempotent}: First off the Corollary gives us a formula for $\omega_p$ as a limit of powers.
\begin{align*}
&([S_p,i_{S_p}] - [S_p,0] ) \times_G \omega_p
\\ ={}& \lim_{n \to \infty} ( [S_p, i_{S_p}] - [S_p,0] ) \times_G ( [S_p, i_{S_p}] - [S_p,0] )^{ (p-1)p^n }
\end{align*}
We have $[S_p,i_{S_p}]-[S_p,0] = [S_p,i_{S_p}]\times_G ([G,i_G]-[G,0])$, and by \eqref{eqRightsideIdempotent} we can push the idempotent $[G,i_G]-[G,0]$ all the way to the end of a long product to get
\begin{align*}
&([S_p,i_{S_p}] - [S_p,0] ) \times_G \omega_p
\\ ={}& \Bigl(\lim_{n \to \infty} ( [S_p, i_{S_p}]_G^G )^{ (p-1)p^n+1 } \Bigr) \times_G ([G,i_G]-[G,0])
\end{align*}
Next we write $[S_p, i_{S_p}]_G^G = ({}_G G_{S_p})\times_{S_p} ({}_{S_p} G_G)$. We can then pull out the initial $({}_G G_{S_p})$ and the final $({}_{S_p} G_G)$, and combine the remain factors in pairs $({}_{S_p} G_G) \times_G ({}_G G_{S_p}) = {}_{S_p} G_{S_p}$. This leads us to
\begin{align*}
&([S_p,i_{S_p}] - [S_p,0] ) \times_G \omega_p
\\ ={}& ({}_G G_{S_p})\times_{S_p}\Bigl(\lim_{n \to \infty} ( {}_{S_p}G_{S_p} )^{ (p-1)p^n } \Bigr)\times_{S_p} ({}_{S_p} G_G) \times_G ([G,i_G]-[G,0])
\\ ={}& ({}_G G_{S_p}) \times_{S_p} (\omega_{\cF_{S_p}(G)}) \times_{S_p} ({}_{S_p} G_G) \times_G ([G,i_G]-[G,0])
\\ ={}& ({}_G G_{S_p}) \times_{S_p} ({}_{S_p} G_G) \times_G ([G,i_G]-[G,0])             \quad\text{ by $\cF_{S_p}(G)$-stability}
\\ ={}& [S_p,i_{S_p}]\times_G ([G,i_G]-[G,0])
\\ ={}& [S_p, i_{S_p}] - [S_p,0].
\end{align*}
This implies that $[S_p, i_{S_p}]_G^G - [S_p,0]_G^G$ is $\omega_p$-stable (not surprisingly).

Now we return to proving Equation \eqref{eqProductInAugmentationIdeal}:
\begin{align*}
&Z \times_G \Bigl( ( [G, i_G] - [G,0] ) - \sum_p \omega_p \Bigr)
\\={}& \Bigl( \sum_p a_p \Bigl( \frac{\lvert G\rvert}{\lvert S_p\rvert} ( [G, i_G] - [G,0] ) - ( [S_p, i_{S_p}] - [S_p,0] ) \Bigr) \Bigr) \times_G \Bigl( ( [G, i_G] - [G,0] ) - \sum_p \omega_p \Bigr)
\\  ={}& \Bigl(\sum_p a_p \frac{\lvert G\rvert}{\lvert S_p\rvert} \Bigr) \cdot \Bigl( ( [G, i_G] - [G,0] ) - \sum_p \omega_p \Bigr)
\\ &- \Bigl( \sum_p a_p ( [S_p, i_{S_p}] - [S_p,0] ) \Bigr) \times_G \Bigl( ( [G, i_G] - [G,0] ) - \sum_p \omega_p \Bigr)
\end{align*}
By Equation \ref{equation}, this is equal to
\begin{align*}
& \Bigl( ( [G, i_G] - [G,0] ) - \sum_p \omega_p \Bigr)
\\ & -  \Bigl( \sum_p a_p ( [S_p, i_{S_p}] - [S_p,0] ) \Bigr) \times_G \Bigl( ( [G, i_G] - [G,0] ) - \sum_p \omega_p \Bigr)
\\ ={}& \Bigl( ( [G, i_G] - [G,0] ) - \sum_p \omega_p \Bigr)
\\ & -  \Bigl( \sum_p a_p \Bigl(( [S_p, i_{S_p}] - [S_p,0] ) \times_G ( [G, i_G] - [G,0] ) - ( [S_p, i_{S_p}] - [S_p,0] ) \times_G \sum_q w_q \Bigr) \Bigr)
\\ ={}& \Bigl( ( [G, i_G] - [G,0] ) - \sum_p \omega_p \Bigr)  -  \Bigl( \sum_p a_p \Bigl(( [S_p, i_{S_p}] - [S_p,0] )  - ( [S_p, i_{S_p}] - [S_p,0] ) \times_G \omega_p \Bigr) \Bigr)
\\ ={}& \Bigl( ( [G, i_G] - [G,0] ) - \sum_p \omega_p \Bigr)  -  \Bigl( \sum_p a_p \cdot 0 \Bigr)
\\ ={}& ( [G, i_G] - [G,0] ) - \sum_p \omega_p.
\end{align*}
Since $Z$ is in $I_G\cdot ([G,i_G]-[G,0])$, \eqref{eqProductInAugmentationIdeal} shows that $( [G, i_G] - [G,0] ) - \sum_p \omega_p$ is in $I_G^k \KG(G,G)$ for all $k$ and therefore equal to $0$ in the $I_G$-adic completion. Thus $( [G, i_G] - [G,0] ) = \sum_p \omega_p$ as we claimed.
\end{example}

\appendix

\section{Categories related to fusion systems}\label{sectionCategories}
We introduce several categories closely connected to the category of fusion systems and study some of the functors between them. We apply the formula for $p$-completion of the previous section to produce a commutative diagram involving these categories.


Let $\G$ be the category of finite groups and group homomorphisms. Fix a prime $p$. Let $\Gs$ be the category with objects pairs $(G,S)$, where $G$ is a finite group and $S$ is a Sylow $p$-subgroup of $G$. A morphism between two objects $(G,S)$ and $(H,T)$ is a homomorphism $f \colon G \rightarrow H$ such that $f(S) \subset T$. Let $\F$ be the category of saturated fusion systems. The objects are saturated fusion systems $(\cF,S)$ and a morphism from $(\cF,S)$ to $(\cG,T)$ is a fusion preserving group homomorphism $S \rightarrow T$.

Recall that $\AG$ is the Burnside category of finite groups. Objects are finite groups and the morphism set between two groups $G$ and $H$, $\AG(G,H)$, is the Grothendieck group of finite $(G,H)$-bisets with a free $H$-action. Also recall that $\AF$ is the Burnside category of fusion systems. Let $\AGs$ be the category with objects pairs $(G,S)$ where $G$ is a finite group and $S$ is a Sylow $p$-subgroup of $G$ and with morphisms between two objects $(G,S)$ and $(H,T)$ given by
\[
\AGs((G,S),(H,T)) = \AG(G,H).
\]

We will also make use of several categories coming from homotopy theory. Let $\Ho(\Top_{\G})$ be the full subcategory of the homotopy category of spaces on the classifying spaces of finite groups. Let $\Ho(\Sp)$ be the homotopy category of spectra and let $\Ho(\Sp_p)$ be the homotopy category of $p$-complete spectra.

The categories above are related by several canonical functors.
The first of these $\A \colon \G \rightarrow \AG$ is the functor from the category of groups to the Burnside category. It takes a group homomorphism $\phi \colon G \rightarrow H$ to the $(G,H)$-biset $[G,\phi]_{G}^{H}$, which is just ${}_G^\phi H_H$ with $G$ acting through $\phi$ on the left. The functor $\A$ is not faithful, conjugate maps are identified. It does factor through the full functor $B(-)$ to $\Ho(\Top_{\G})$ and $\Ho(\Top_{\G})$ does map faithfully into $\AG$.

Let $U \colon \Gs \rightarrow \G$ be the forgetful functor, sending $(G,S)$ to $G$. This functor is faithful. While the functor is not full, given a map $\phi \colon G \rightarrow H$ and a Sylow subgroup $S \subseteq G$, $\phi(S)$ is contained in some Sylow subgroup of $H$ and this provides a lift of $\phi$ to $\Gs$.

Let $\As \colon \Gs \rightarrow \AGs$ be defined just as the functor $\A$. It takes a morphism $\phi \colon (G,S) \rightarrow (H,T)$ to the biset $[G, \phi]_{G}^{H}$. Note that any homomorphism $G \rightarrow H$ is $H$-conjugate to a homomorphism sending $S$ into $T$. Thus the image of $\Gs((G,S), (H,T))$ in
\[
\AGs((G,S),(H,T)) = \AG(G,H)
\]
is equal to the image of $\G(G,H)$ under the functor $\A$.

Let $\AU \colon \AGs \rightarrow \AG$ be the forgetful functor. Note that this functor is an equivalence, it is fully faithful and surjective on objects.

Let $F \colon \Gs \rightarrow \F$ be the functor sending a pair $(G,S)$ to the induced fusion system $\cF_G$ on $S$. By construction, the morphisms in $\Gs$ restrict to fusion preserving maps between the chosen Sylow $p$-subgroups.

Let $\Af \colon \F \rightarrow \AF$ be the functor that is the identity on objects and takes a map of fusion systems $(\cF,S) \rightarrow (\cG,T)$ induced by a fusion preserving map $\phi \colon S \rightarrow T$ to $[S, \phi]_{\cF}^{\cG}$. In Proposition \ref{ApFunctor} below, we prove that this is a functor.

\begin{lemma}\label{fusionPreservingMaps}
Let $\phi\colon S\to T$ be a fusion preserving map between saturated fusion systems $(\cF,S)$ and $(\cG, T)$. Then
\[\omega_\cF\times_S [S,\phi] \times_T \omega_\cG = [S,\phi]\times_T \omega_\cG.\]
\end{lemma}

\begin{proof}
It is sufficient to show that $X:= [S,\phi]\times_T \omega_\cG$ is left $\cF$-stable. To see that $X$ is $\cF$-stable we consider an arbitrary subgroup $P\leq S$ and map $\psi\in \cF(P,S)$ and prove that the restriction of $X$ along $\psi$, ${}_P^\psi X_T$, is isomorphic to ${}_P X_T$ as virtual $(P,T)$-bisets.

The restriction of $[S,\phi]_S^T$ along $\psi\colon P\to S$ is just $[P,\phi\circ \psi]_P^T$. We therefore have
\[{}_P^\psi X_T = [P,\phi\circ \psi]_P^T \times_T \omega_\cG.\]
Since $\psi$ is a map in $\cF$, and since $\phi$ is assumed to be fusion preserving, this means that there is some map $\rho\colon \phi(P) \to T$ in $\cG$ such that $\phi|_{\psi(P)}\circ \psi = \rho\circ \phi|_P$. Finally, $\omega_\cG$ absorbs maps in $\cG$, and thus
\[{}_P^\psi X_T = [P,\phi\circ \psi]_P^T \times_T \omega_\cG = [P,\rho\circ \phi]_P^T \times_T \omega_\cG =  [P,\phi]_P^T \times_T \omega_\cG = {}_P X_T.\qedhere\]
\end{proof}

\begin{proposition}\label{ApFunctor}
The operation $\Af$ described above is a functor.
\end{proposition}

\begin{proof}
Suppose we have two fusion preserving maps $\psi\colon R\to S$, $\phi\colon S\to T$ between saturated fusion systems $(\cE,R)$, $(\cF,S)$ and $(\cG,T)$. Applying Lemma \ref{fusionPreservingMaps} to $\phi$, we easily confirm that $\Af$ preserves composition:
\begin{align*}
\Af(\phi) \circ \Af(\psi) &= \omega_\cE \times_R [R,\psi] \times_S \omega_\cF \times_S [S,\phi] \times_T \omega_\cG
\\ &= \omega_\cE \times_R [R,\psi] \times_S [S,\phi] \times_T \omega_\cG
\\ &= \omega_\cE \times_R [R,\phi\circ \psi] \times_T \omega_\cG
\\ &= \Af(\phi\circ \psi).\qedhere
\end{align*}
\end{proof}

Let $\widehat{(-)} \colon \AGs \rightarrow \AF$ be the functor taking $(G, S)$ to $\cF_{G}$ and taking a virtual $(G,H)$-biset $X$ to
\[
{}_{\cF_G}\widehat{X}_{\cF_H}
\]
as described in Theorem \ref{mainthm}.

\begin{proposition} \label{comparegroupsandburnside}
There is a commutative diagram of categories
\[
\xymatrix{\G \ar[r]^{\A}  & \AG \\ \Gs \ar[u]^{U} \ar[r]^{\As} \ar[d]_{F} & \AGs \ar[u]_{\A U} \ar[d]^{\widehat{(-)}} \\ \F \ar[r]_{\Af} & \AF.}
\]
\end{proposition}
\begin{proof}
The top square clearly commutes, so we need to prove that the bottom square commutes as well.

Let $\phi\colon G\to H$ be a morphism in $\Gs$ from $(G,S)$ to $(H,T)$, meaning that $\phi(S)\leq T$. The functor $\As$ takes $\phi$ to the biset $[G,\phi]_G^H$. To understand what happens when we apply $\widehat{(-)}$ to $[G,\phi]_G^H$, we first need to understand the restriction ${}_S ([G,\phi]_G^H)_T$ of the biset to the Sylow $p$-subgroups.

We can describe the restriction ${}_S ([G,\phi]_G^H)_T$ as composing with the inclusion biset ${}_S G_G$ on the left and the transfer biset ${}_H H_T$ on the right:
\[{}_S ([G,\phi]_G^H)_T = {}_S G \times_G [G,\phi]_G^H \times_H H_T.\]
Now ${}_S G \times_G [G,\phi]_G^H$ simply gives us the restriction of $\phi$ to the subgroup $S$, $[S,\phi|_S]_S^H$. By assumption $\phi|_S$ lands in $T\leq H$, and therefore
\[{}_S ([G,\phi]_G^H)_T = [S,\phi|_S]_S^H \times_H H_T = [S,\phi|_S]_S^T \times_T H\times_H H_T = [S,\phi|_S]_S^T \times_T H_T.\]
The biset ${}_T H_T$ is $\cF_H$-stable and invertible inside $\AF(\cF_{H},\cF_{H})$, and the functor $\widehat{(-)}$ applied to $[G,\phi]_G^H$ is by Theorem \ref{mainthm} equal to
\begin{align*}
\widehat{ \As ( \phi)} &= {}_S ([G,\phi]_G^H) \times_T ({}_{\cF_H} H_{\cF_H})^{-1}
\\ &= [S,\phi|_S]_S^T \times_T H \times_T ({}_{\cF_H} H_{\cF_H})^{-1}
\\ &= [S,\phi|_S]_S^T \times_T \omega_{\cF_H}
\\ &= [S,\phi|_S]_{\cF_G}^{\cF_H} = \Af(F(\phi))
\end{align*}
The penultimate equality is due to Lemma \ref{fusionPreservingMaps} since the restriction $\phi|_S$ is a fusion preserving map from $\cF_G$ to $\cF_H$.
\end{proof}

Let $\al \colon \AGs \rightarrow \Ho(\Sp)$ be the functor sending $G$ to $\Sinfp BG$ and sending $[K, \phi]_{G}^{H}$ to the composite
\[
\Sinfp BG \lra{\Tr_{G}^{K}} \Sinfp BK \lra{\Sinfp B\phi} \Sinfp BH,
\]
where $\Tr$ is the transfer. This functor is well-understood by the solution to the Segal conjecture. It is neither full nor faithful.

Let $\beta \colon \AF \rightarrow \Ho(\Sp_p)$ be the analogous functor for fusion systems. It sends the object $\cF$ to $\Sinfp B \cF$ and applies Proposition \ref{aniso} to maps. The functor $\beta$ is fully faithful.

Let $(-)^{\wedge}_{p}$ be the $p$-completion functor $\Ho(\Sp) \rightarrow \Ho(\Sp_p)$.

\begin{proposition}
There is a commutative diagram
\[
\xymatrix{\AG \ar[r]^-{\al}& \Ho(\Sp) \ar[dd]^-{(-)^{\wedge}_{p}} \\ \AGs \ar[u]^-{\AU}_-{\simeq}  \ar[d]_-{\widehat{(-)}} &  \\ \AF \ar[r]^-{\beta} & \Ho(\Sp_p)}
\]
up to canonical natural equivalence and the formula for $\widehat{(-)}$ is given by Theorem \ref{mainthm}.
\end{proposition}
\begin{proof}
This is a consequence of Theorem \ref{mainthm}.
\end{proof}

Let $\AG^{\mathrm{free}}(G,H)$ be the submodule of the Burnside module $\AG(G,H)$ generated by bisets that have a free action by both $G$ and $H$. This submodule has a basis consisting of isomorphism classes of sets of the form $[K,\phi]_{G}^{H}$, where $\phi$ is an injection. This includes the biset  ${}_G H_H = [G,i]_{G}^{H} = \A_{\mathrm{syl}}(i)$, where $G$ acts on $H$ through an injection $i \colon G \to H$.

We have a natural isomorphism
\[
(-)^{\mathrm{op}} \colon \AG^{\mathrm{free}}(G,H) \xrightarrow{\cong} \AG^{\mathrm{free}}(H,G) \subset \AG(H,G)
\]
sending
\[
{}_G X_H \mapsto ({}_G X_H)^{\mathrm{op}} = {}_H (X^{\mathrm{op}})_G,
\]
where $X^{\mathrm{op}}$ has the same underlying set as $X$, and an element $g\in G$ acts on an element $x\in X$ from the right by $g^{-1}\cdot x$ using the left $G$-action on $X$. Similarly $H$ has a left-action on $X^{\mathrm{op}}$.

Thus elements of $\AG^{\mathrm{free}}(G,H)$ give rise to maps in $\mathbb{AG}$ not only from $G$ to $H$, but applying $(-)^{\mathrm{op}}$ we also get a map in $\AG$ from $H$ to $G$. The image under $(-)^{\mathrm{op}}$ of $[G,i]_{G}^{H} = {}_G H_H$, where $i$ is an injection, is referred to as the transfer map along $i$ and is given by the biset $({}_G H_H)^{\mathrm{op}} = {}_H H_G = [i(G), i^{-1}]_H^G$.

The same story makes sense for fusion preserving injections between two fusion systems. Let $\cF_1$ and $\cF_2$ be saturated fusion systems on $p$-groups $S_1$ and $S_2$, and denote by $\AF^{\mathrm{free}}(\cF_1,\cF_2)$ the collection of $(\cF_1,\cF_2)$-stable element in $\AG^{\mathrm{free}}(S_1,S_2)^{\wedge}_p$. The isomorphism
\[
(-)^{\mathrm{op}} \colon \Z_p \otimes \AG^{\mathrm{free}}(S_1,S_2) \xrightarrow{\cong} \Z_p \otimes \AG^{\mathrm{free}}(S_2,S_1)
\]
induces an isomorphism
\[
(-)^{\mathrm{op}} \colon \AF^{\mathrm{free}}(\cF_1,\cF_2) \xrightarrow{\cong} \AF^{\mathrm{free}}(\cF_2,\cF_1).
\]
A transfer map from $\cF_2$ to $\cF_1$ is the image of an element of the form $[S_1,i]_{\cF_1}^{\cF_2} = \Af(i)$, where $i$ is an injection of fusion systems, under the map $(-)^{\mathrm{op}}$.

Since the functor $F$ takes injection to injections, Proposition \ref{comparegroupsandburnside} implies that
\begin{equation} \label{lastequation}
\widehat{[G,i]_{G}^{H}} = [S,F(i)]_{\cF_G}^{\cF_H},
\end{equation}
where $S \subset G$ is a Sylow $p$-subgroup. Thus the functor $\widehat{(-)}$ ``takes injections to injections.''

It is tempting to assume that the functor $\widehat{(-)}$ will also preserve transfer maps. However, in general,
\[
(\widehat{{}_G H_H})^{\mathrm{op}} \neq \widehat{{}_H H_G}.
\]
To see this, let $G = e$ and let $H$ to be a non-trivial group of order prime to $p$. Consider the element $[e,i]_{e}^{H} \in \AG(e,H)$, where $i$ is the inclusion of the identity element. Composing this biset with its opposite, $[e,i]_H^e$, we get the element ${}_{e}H_{e} = |H| \in \AG(e,e)\cong \Z$. Since both $e$ and $H$ have trivial Sylow $p$-subgroups, \eqref{lastequation} implies that
\[
\widehat{_{e}H_{H}} = \Id_{\cF_e}
\]
and thus
\[
(\widehat{_{e}H_{H}})^{\mathrm{op}} = ( \Id_{\cF_e})^{\mathrm{op}} =  \Id_{\cF_e}.
\]
However,
\[
\widehat{_{H}H_{e}} \circ \widehat{_{e}H_{H}}  = \widehat{_{e}H_{e}} = |H| \in \AF(e,e) \cong \Z_p.
\]
Since $\widehat{_{e}H_{H}} = \Id_{\cF_e}$, we conclude that the operations $(-)^{\mathrm{op}}$ and $\widehat{(-)}$ do not commute:
\[
\widehat{(_{e}H_{H})^{\mathrm{op}}} = \widehat{_{H}H_{e}}  \neq \Id_{\cF_e} = \bigl(\widehat{_e H_H}\bigr)^{\mathrm{op}}.
\]
It may come as a surprise to the reader to find out that confusion regarding this issue, the relationship between $p$-completion and transfers, was the original motivation for this paper.

\end{document}